\documentclass[reqno,final]{amsart}

\usepackage{graphicx}
\usepackage{amsmath, amsthm,  amsfonts, amscd}
\usepackage{url}
\usepackage[color,notref,notcite]{showkeys} 
\usepackage[normalem]{ulem} 
\usepackage{esint} 
\usepackage{enumerate} 

\def\XXint#1#2#3{{\setbox0=\hbox{$#1{#2#3}{\int}$ }
\vcenter{\hbox{$#2#3$ }}\kern-.6\wd0}}

\def\ol#1{\overline{#1}}

\newcommand*{\mathcolor}{}
\def\mathcolor#1#{\mathcoloraux{#1}}
\newcommand*{\mathcoloraux}[3]{%
  \protect\leavevmode
  \begingroup
    \color#1{#2}#3%
  \endgroup
} 

\theoremstyle{plain}
\newtheorem{theorem}{Theorem}[section]

\newtheorem{lemma}[theorem]{Lemma}

\newtheorem{proposition}[theorem]{Proposition}

\newtheorem{corollary}[theorem]{Corollary}

\theoremstyle{definition}
\newtheorem{definition}[theorem]{Definition}

\def\bhyp#1{\begin{equation}\label{#1}\begin{gathered}}
\def\ehyp{\end{gathered}\end{equation}}

\newcounter{cst}
\def \ctel#1{C_{\refstepcounter{cst}\label{#1}\thecst}}
\def \cter#1{C_{\ref{#1}}}

\theoremstyle{remark}
\newtheorem{remark}{Remark}[section]

\numberwithin{equation}{section}
\numberwithin{figure}{section}

\definecolor{violet}{rgb}{0.580,0.,0.827}

\newcommand{\RR}{{\mathbb R}}
\newcommand{\NN}{{\mathbb N}}

\newcommand{\cM}{{\mathcal M}}
\newcommand{\distt}{\mathcal{D'}(\Omega\times(0,T))} 

\DeclareMathOperator{\dive}{\mathrm{div}} 

\DeclareMathOperator{\meas}{\mathrm{meas}} 

\newcommand{\ud}{\, \mathrm{d}} 

\newcommand{\supp}[1]{\mathrm{supp}\,#1} 

\newcommand{\U}{\mathbf{u}} 

\newcommand{\K}{\mathbf{K}} 

\newcommand{\tK}{\widetilde{\mathbf{K}}} 

\newcommand{\ch}{\mathbf{1}} 

\newcommand{\dm}{d_m}
\newcommand{\dl}{d_l}
\newcommand{\dt}{d_t}
\renewcommand{\O}{\Omega}
\newcommand{\eps}{\varepsilon}
\newcommand{\discK}{\mathfrak{D}_\K}
\newcommand{\bpsi}{\boldsymbol{\psi}}

\newcommand{\weakto}{\rightharpoonup}

\newcommand{\norm}[1]{\left\lVert#1\right\rVert} 
\newcommand{\seminorm}[1]{\left\lvert#1\right\rvert} 

\newcommand{\gradc}[2]{\nabla_{\! \!\{ #1 > #2\}}c} 

\newcommand{\gradf}[2]{\nabla_{\! \!\{ #1 > #2\}}f} 

\newcommand{\D}[3][ ]{\mathbf{D}_{#1}(#2, #3)} 

\newcommand{\Dsq}[3][ ]{\mathbf{D}_{#1}^{1/2}(#2, #3)} 

\newcommand{\DD}{\mathbf{D}}

\newcommand{\Leb}[3][ ]{L^{#2}(0, T; L^{#3}(\Omega)^{#1})} 

\newcommand{\Tprss}[3]{p_{#1}^{#2,#3}} 
\newcommand{\Tdarcy}[3]{\mathbf{u}_{#1}^{#2,#3}} 
\newcommand{\Tconc}[3]{c_{#1}^{#2,#3}} 

\newcommand{\prss}[2][ ]{p_{#2}^{#1}}
\newcommand{\darcy}[2][ ]{\mathbf{u}_{#2}^{#1}}
\newcommand{\conc}[2][ ]{c_{#2}^{#1}}

\newcommand{\nuapprox}{\nu_{n}}
\newcommand{\aapprox}{a_{n}}

\title[Miscible displacement with vanishing diffusion]{Analysis of miscible displacement 
through porous media with vanishing molecular diffusion and singular wells}

\author{J\'er\^ome Droniou}
\address[J\'er\^ome Droniou]{School of Mathematical Sciences, Monash University, Victoria 3800, Australia.}
\email{jerome.droniou@monash.edu}

\author{Kyle S. Talbot}
\address[Kyle S. Talbot]{School of Mathematical Sciences, Monash University, Victoria 3800, Australia.}
\email{kyle.talbot@monash.edu}

\thanks{The first author was supported in part by the Australian Research Council's Discovery Projects funding scheme
(project number DP170100605). The second author was supported in part by an Endeavour Research Fellowship from the Australian Government.}

\date{\today}

\begin{document}


\begin{abstract}
This article proves the existence of solutions to a model of incompressible 
miscible displacement through a porous medium, with zero molecular diffusion
and modelling wells by spatial measures.
We obtain the solution by passing to the limit on problems indexed by vanishing
molecular diffusion coef\-ficients. The proof employs cutoff functions to excise 
the supports of the measures and the discontinuities in the permeability tensor, 
thus enabling compensated compactness arguments used by Y. Amirat and A. Ziani for the analysis 
of the problem with $L^2$ wells [\emph{Z. Anal. Anwendungen}, 23(2):335--351, 2004]. 
We give a novel treatment of the diffusion-dispersion term, which
requires delicate use of the Aubin--Simon lemma to ensure the strong convergence of
the pressure gradient, owing to the troublesome lower-order terms introduced by the localisation 
procedure.
\end{abstract}


\maketitle


\section{Introduction} \label{sec:intro}

\subsection{The miscible displacement problem} \label{ssec:model}
We study the single-phase, miscible displacement of one incompressible fluid by another 
through a porous medium, as occurs in enhanced oil recovery processes. 
Neglecting gravity, the model reads \cite{ew83,pe77}
\begin{subequations}\label{eq:model}
\begin{equation}\left.
	\begin{aligned}
		&\U(x,t) = -\frac{\K(x)}{\mu(c(x,t))}\nabla p(x,t)\\
		&\dive \U(x,t) = (q^I - q^P)(x,t)
	\end{aligned}
	\right\}, \quad(x,t)\in\O\times(0,T),\label{eq:darcy}
\end{equation} 
\begin{multline}
\Phi(x)\partial_{t}c(x, t) - \dive\big(\D{x}{\U(x,t)} \nabla c - c\mathbf{u}\big)(x, t) + (q^{P}c)(x, t) = (q^{I}\hat{c})(x, t),\\
\mbox{\quad $(x, t) \in\Omega \times (0, T)$,}\label{eq:conc}
\end{multline}
subject to the no-flow boundary conditions
\begin{align}
\U(x, t) \cdot \mathbf{n} &= 0, \mbox{\quad $(x, t) \in \partial \Omega \times (0, T)$, and} \label{eq:bc1}\\
\D{x}{\U(x,t)}\nabla c(x,t) \cdot \mathbf{n} &= 0, \mbox{\quad $(x, t) \in \partial\Omega \times (0, T)$,} \label{eq:bc2}
\end{align}
the initial condition
\begin{align}
c(x,0) = c_{0}(x), \mbox{ \quad $x \in \Omega$,} \label{eq:ic}
\end{align}
and a normalisation condition to eliminate arbitrary constants in the solution 
$p$ of the elliptic equation \eqref{eq:darcy}:
\begin{equation}
\int_{\Omega} p(x, t) \ud x = 0 \mbox{\quad for all $t \in (0, T)$.} \label{eq:normp}
\end{equation}

The unknowns of the system are the pressure $p$ and Darcy velocity $\U$ of the 
fluid mixture, and the concentration $c$ of one of the components in the fluid 
mixture. The reservoir is represented by $\O$, a bounded connected open 
subset of $\RR^d$, $d=2$ or $3$, and the recovery process occurs over the time interval $(0, T)$. 
The reservoir-dependent quantities of porosity and absolute 
permeability are $\Phi$ and $\K$, respectively. We denote by $q^I$ and $q^P$ the sums of injection 
well source terms and production well sink terms (henceforth collectively referred to
as source terms), respectively, and write $\hat{c}$ for the concentration of the injected fluid. 

The coefficient $\mathbf{D}$ in \eqref{eq:conc} is the diffusion-dispersion 
tensor, derived by Peaceman \cite{pe66} as 
\begin{equation}
\mathbf{D}(x, \mathbf{u}) = \Phi(x)\bigg(d_{m}\mathbf{I} + |\mathbf{u}|\Big(d_{l}E(\mathbf{u}) + d_{t}(\mathbf{I} - E(\mathbf{u}))\Big)\bigg), \label{eq:ddt}
\end{equation}
where
\begin{equation}
E(\mathbf{u}) = \left(\frac{\mathbf{u}_i \mathbf{u}_j}{|\mathbf{u}|^{2}} \right)_{1 \leq i,j \leq d} \label{eq:proj}
\end{equation}
is the projection in the direction of flow. The constants $d_{m}, d_{l}$ and 
$d_{t}$ are the mole\-cular diffusion coefficient and
the longitudinal and transverse mechanical dispersion coefficients, respectively. 
After Koval \cite{kov63} (see also \cite{ce99,rw83}), the concentration-dependent viscosity $\mu$ of the fluid 
mixture often assumes the form
\begin{equation}
\mu(c) = \mu(0)\left(1 + \big(M^{1/4} - 1\big)c\right)^{-4} \mbox{\quad for $c \in [0, 1]$,} \label{eq:visc}
\end{equation}
where the mobility ratio $M:=\frac{\mu(0)}{\mu(1)}>1$.
Finally, the boundary condition \eqref{eq:bc1} enforces a compatibility condition 
upon the source terms:
\begin{equation}
\int_{\Omega} q^{I}(x, t) \ud x = \int_{\Omega} q^{P}(x, t) \ud x 
\mbox{\quad for all $t \in (0, T)$.} \label{eq:compat} 
\end{equation}
\end{subequations}

\subsection{Principal contributions} \label{ssec:contribs}
Our main result, Theorem \ref{th:main}, is the existence of weak solutions to \eqref{eq:model}
when $\dm=0$ and $q^I$ and $q^P$ are modelled spatially as bounded, nonnegative Radon measures
on $\O$. Indeed, the novelty of this article is the presence of both these features
simultaneously; Amirat and Ziani \cite{az04} analyse the system as $\dm\to0$ with
$q^I$, $q^P\in\Leb{\infty}{2}$, and our previous work \cite{dt14} establishes
existence for $\dm>0$ and measure source terms. Fabrie and Gallou\"et \cite{fg00} 
assume that the diffusion-dispersion tensor is uniformly bounded to address
the latter scenario. The first existence result for \eqref{eq:model} as written above
is due to Feng \cite{fe95}, focussing mostly on the two-dimensional problem 
with sources in $\Leb{\infty}{2}$. The subsequent analysis of Chen and Ewing \cite{ce99} 
is valid for very general boundary conditions in three dimensions, but assumes
$d_m>0$ and regular source terms. Uniqueness is known for ``strong'' solutions \cite{fe95},
but appears to be open for weak solutions even with $\dm>0$ fixed \cite{az04,ce99,fe95}.

We prove Theorem \ref{th:main} by passing to the limit as $\dm\to0$ on a sequence of
problems with measure source terms defined in Section \ref{ssec:measapprox}. In further contrast to Amirat--Ziani
who take $\Phi\equiv1$ and $\K$ continuous, we only assume that the porosity is bounded,
and we allow for discontinuous permeabilities of the kind that one expects in practice \cite{dsws94}.
Working in such a low-regularity environment leads to the challenge of identifying
the limits of the nonlinear terms $-\frac{\K}{\mu(c)}\nabla p$ and $\D{\cdot}{\U}\nabla c$ as $\dm\to0$. 
For this task we use smooth cutoff functions --- first appearing in Section \ref{ssec:measest}
--- to excise both the supports of the
measures and the discontinuities in $\K$, thereby localising the problem to where
the data is smooth enough for us to employ a compensated compactness-type lemma
(Lemma \ref{lem:cc}). 

This localisation procedure nonetheless introduces problems of its own in the form of
lower-order terms that inhibit a straightforward proof of strong convergence of the 
pressure gradients, as is the case for $L^2$ sources. We handle these lower-order 
terms by exploiting the uniqueness of the solution to the elliptic problem in combination with
careful use of the Aubin--Simon compactness lemma to first prove strong 
convergence of the pressure itself in Section \ref{ssec:strongconvU}. 

Strong convergence of the pressure gradients (and then the Darcy velocities) is 
crucial for our treatment of the diffusion-dispersion term $\D{\cdot}{\U}\nabla c$
in Section \ref{ssec:passlimitconc}, which we believe is also novel. 
In particular, we fill a gap in the work of Amirat--Ziani by giving meaning to 
$\nabla c$ in the limit as $\dm\to0$. When the molecular diffusion
is neglected, the concentration gradient is only well-defined as a function in 
non-stagnant zones of the reservoir; that is, where $\U\neq\mathbf{0}$. 
We introduce a new notion in Section \ref{ssec:graddef}
that resolves this difficulty.

\subsection{Why vanishing molecular diffusion and singular wells?} 

The interest in studying \eqref{eq:model} with $\dm=0$ is twofold. In practice,
the mechanical dispersion coefficients will be at least an order of
magnitude larger than $\dm$, so the effects of molecular diffusion are negligible
compared to those of mechanical dispersion \cite{anderson84, pt02, young84}. 
Moreover, in practical simulations of \eqref{eq:model} the mesh size is such
that the effects of molecular diffusion are dominated by numerical diffusion, so
$\dm$ is often neglected from the simulation \cite{rw83,bs84}.

Scale differences motivate the decision to model $q^I$ and $q^P$ as measures. 
The diameter of typical reservoir ($\sim10^3$m) is several orders of magnitude larger than 
that of a typical wellbore ($\sim10^{-1}$m). At field scale the wells are thus 
effectively point (resp. line) sources in two (resp. three) dimensional models.



\section{Assumptions and main result} \label{sec:assump}

\subsection{Assumptions on the data} \label{ssec:assump}

We make the following assumptions on the data:
\begin{subequations}\label{assum:general}
\bhyp{hyp:domain}
T \in \RR_{+}^{\ast} \mbox{ and }\Omega \mbox{ is a bounded, connected, open subset of }\RR^{d}, d\leq3, \\
\mbox{ with a Lipschitz continuous boundary}.
\ehyp
Writing $\discK$ for the closure of the set of discontinuities of $\K$, we assume that $\discK$
has zero Lebesgue measure (in practice, $\discK$ is contained in a finite union of
hypersurfaces). Write $\mathcal{S}_d(\RR)$ for the 
set of $d\times d$ symmetric matrices. The permeability satisfies
\bhyp{hyp:K}
\mathbf{K}: \Omega \to \mathcal{S}_{d}(\RR) \mbox{ is locally Lipschitz continuous on $\O \setminus \discK$, and}\\
\exists k_{\ast}>0 \mbox{ such that, for a.e.\ }x \in \Omega \mbox{ and for all } \xi \in \RR^{d},\\
k_{\ast}|\xi|^{2} \leq \mathbf{K}(x)\xi \cdot \xi \leq k_{\ast}^{-1}|\xi|^{2}.
\ehyp
The porosity $\Phi$ is such that
\bhyp{hyp:porosity}
\Phi \in L^{\infty}(\Omega) \mbox{ and there exists } \phi_{\ast} > 0 \mbox{ such that for a.e.\ $x \in \Omega$, }\\
\phi_{\ast} \leq \Phi(x) \leq \phi_{\ast}^{-1}.
\ehyp
Particularly important to our analysis are the assumptions on the viscosity:
\bhyp{hyp:viscosity}
\mu \in C^{2}([0,1]; (0,\infty)) \mbox{ is such that $\mu''>0$ and $(1/\mu)''>0$.}\\
\mbox{We write $\mu_\ast$ and $\mu^\ast$ for the minimum and maximum of $\mu$, respectively.}
\ehyp
This implies the strict convexity of $\mu$ and $1/\mu$. 
Note that the form \eqref{eq:visc} satisfies \eqref{hyp:viscosity}.
By setting $d_m = 0$ in \eqref{eq:ddt}, we introduce the mechanical dispersion tensor
\begin{equation}
\D[\circ]{x}{\U} = \Phi(x)|\U|\Big(d_{l}E(\U) + d_{t}(\mathbf{I} - E(\U))\Big), \label{eq:mdt}
\end{equation}
and note that it satisfies
\bhyp{hyp:mdt}
\mathbf{D}_\circ: \O\times\RR^d \to \mathcal{S}_{d}(\RR) 
\mbox{ is a Carath\'eodory function such that for a.e.\ $x \in \O$}\\
\mbox{and for all $\zeta,\xi\in\RR^{d}$, }
\D[\circ]{x}{\zeta}\xi\cdot\xi \geq \phi_\ast \min(\dl,\dt)|\zeta||\xi|^{2} \mbox{ and }\\ 
|\D[\circ]{x}{\zeta}\xi|\leq\phi_{\ast}^{-1}\max(\dl,\dt)|\zeta||\xi|.
\ehyp
The injected and initial concentration are such that
\bhyp{hyp:injectedconc}
\hat c \in L^{\infty}(0,T; C(\ol{\Omega})) \mbox{ satisfies } 0 \leq \hat c(x, t) \leq 1 \mbox{ for a.e.\ }(x, t) \in \Omega \times (0, T),
\ehyp
\bhyp{hyp:initialconc}
c_{0} \in L^{\infty}(\Omega) \mbox{ satisfies } 0 \leq c_{0}(x) \leq 1 \mbox{ for a.e.\ } x \in \Omega.
\ehyp
The source terms are such that
\bhyp{hyp:source}
q^{I} = a\nu \mbox{ and }q^{P} = b\nu, \mbox{ where }\\
a, b \in L^\infty(0,T; C(\ol{\O}))\mbox{ are nonnegative on $\O\times(0,T)$,}\\
\nu\in \cM_+(\Omega)\cap (W^{1,\ell}(\O))'\mbox{ for all $\ell>2$,}\\
\mbox{and $\supp(\nu)$ has zero Lebesgue measure.}
\ehyp
Here $\cM_{+}(\Omega)$ is the set of bounded nonnegative Radon measures on $\Omega$.
The compatibility condition imposed by \eqref{eq:bc1} becomes
\begin{equation}
\int_{\Omega}a(x, t)\ud \nu(x) = \int_{\Omega}b(x, t)\ud \nu(x) \quad \forall t \in (0, T). \label{hyp:compatibility}
\end{equation}
\end{subequations}

\begin{remark} 
We impose the condition $\nu\in (W^{1,\ell}(\O))'$ for all $\ell>2$ in order to employ
a sharp uniqueness result for the elliptic equation with measure data. 
This uniqueness result --- which compensates for the absence of estimates on $\partial_{t}p$ ---
is instrumental to establishing the strong convergence of the pressure.
This $(W^{1,\ell}(\O))'$ regularity is satisfied by all measures in two dimensions,
and by all measures that can reasonably be used to model wells in three dimensions;
see \cite{fg00}.
\end{remark}

For a topological vector space $X(\Omega)$ of functions on $\Omega$, we write 
$(X(\Omega))'$ for its topological dual. When writing the duality pairing
$\langle \cdot , \cdot \rangle_{(X(\Omega))', X(\Omega)}$, we omit the spaces if
they are clear from the context.
When $z\in(1, \infty)$ is a Lebesgue exponent, we write $z'=\frac{z}{z-1}$ for its conjugate. 
We denote by $W^{1,z}_\star(\O)$ those elements of $W^{1,z}(\O)$ whose integral
over $\O$ vanishes.
For $k \in \RR$ and $g: \O \to \RR$, we denote by $\{ g = k \}$ the level set 
$\{ x \in \O \, | \, g(x) = k \}$;
similarly for sublevel sets $\{ g \leq k \}$, $\{ g < k \}$ and superlevel sets 
$\{ g \geq k \}$, $\{ g > k \}$.
When a constant appears in an estimate we track only its relevant dependencies.
In particular, we do not indicate dependencies with respect to $\phi_\ast$, $d_l$, 
$d_t$, $T$, $\O$, $k_\ast$, $\mu_\ast$, $\mu^\ast$ or $\hat{c}$, as these quantities 
remain constant throughout the paper.
When stating that a certain constant depends only on some quantity $X$,
it is implicitly understood that this dependency is nondecreasing.

Before detailing our results, we must first introduce a new concept that is key to
our notion of solution when $d_m = 0$.


\subsection{The concentration gradient in the absence of molecular diffusion}
\label{ssec:graddef}

Consider $\dm=\eps>0$. Write $(p_\eps,\U_\eps,c_\eps)$ for the corresponding
solution to \eqref{eq:model} (the existence of which we discuss shortly), and
$\D[\eps]{\cdot}{\U_\eps}$ the corresponding diffusion-dispersion tensor. A straightforward
computation using the definition \eqref{eq:ddt} shows that
\begin{equation*}
\int_0^T\int_\O\D[\eps]{\cdot}{\U_\eps}\nabla c_\eps\cdot\nabla c_\eps
\geq \eps\int_0^T\int_\O|\nabla c_\eps|^2
+ \min(\dl,\dt)\int_0^T\int_\O|\U_\eps||\nabla c_\eps|^2.
\end{equation*}
Thus, in order to obtain estimates on $\nabla c_\eps$ as $\eps\to0$, it seems necessary
to first restrict attention to regions where $|\U_\eps|>\eta>0$. This leads to the
following definition, which we use in the treatment of the diffusion-dispersion
term to give meaning to $\nabla c$ in the limit as $\dm\to0$.
\begin{definition} \label{def:ugrad}
Let $f, v \in \Leb{2}{2}$, with $v \geq 0$. We say that $f$ has a \emph{$\{ v > 0\}$-gradient} 
if
\begin{itemize}
\item there are sequences $(f_\eps)_{\eps>0}$ in $L^2(0,T;H^1(\O))$ and 
$(v_\eps)_{\eps>0}$ in $L^1(\O\times(0,T))$ such that as $\eps\to0$,  
	\begin{align*}
	f_\eps &\rightharpoonup f \quad \mbox{weakly in $\Leb{2}{2}$,}\\
	v_\eps &\to v \quad \mbox{a.e.\ on }\O \times (0,T);
	\end{align*}
\item there is a sequence $(\eta_i)_{i\in\NN}$ in $\RR$ with $\eta_i\to 0^{+}$ as $i\to\infty$ such that 
for every $i\in\NN$, $\meas(\{v=\eta_i\})=0$, and for some function $\chi_{\eta_i} \in \Leb[d]{2}{2}$,
\begin{align*}
\mathbf{1}_{\{ v_\eps > \eta_i \}}\nabla f_\eps \rightharpoonup \chi_{\eta_i } \quad 
\mbox{weakly in $\Leb[d]{2}{2}$ as }\eps\to0.
\end{align*}
\end{itemize}
We then denote $\gradf{v}{\eta_{i}} := \chi_{\eta_i}$ the 
\emph{$\{ v > \eta_i \}$-gradient of $f$} and define the \emph{$\{ v > 0 \}$-gradient of $f$} 
as the function $\gradf{v}{0}$ satisfying
\begin{equation*}
\gradf{v}{0} = 
	\begin{cases} 
	\gradf{v}{\eta_i} & \text{on }\{ v > \eta_i \} \quad\forall i\in\NN,\\
	0 & \text{on }\{ v = 0 \}.
	\end{cases}
\end{equation*}
\end{definition}
Appendix \ref{sec:appa} establishes some important properties that this construction satisfies.

\begin{remark}\label{rem:freg} If $f$ is a regular function then
$\gradf{v}{0}=\nabla f$ on $\{v>0\}$.
\end{remark}


\subsection{Main result} \label{ssec:results}

The principal contribution of this article is the following existence result.
\begin{theorem} \label{th:main}
Under Hypotheses \eqref{assum:general}, there exists a weak solution $(p,\U,c)$
to \eqref{eq:model} with $d_m=0$ in the following sense:
\begin{subequations}\label{def:sol.meas}
\begin{equation} \label{eq:measc}
	\begin{gathered}
	c\in L^\infty(\O\times(0,T))\,,\;0 \leq c(x,t) \leq 1 \mbox{\quad for a.e.\ $(x,t)\in\O\times(0,T)$,} \\
	c \in L^{\infty}(0,T; L^{1}(\Omega, \nu)), \\
	0 \leq c(x,t) \leq 1 \mbox{\quad for $\nu$-a.e.\ $x \in \Omega$, for a.e.\ $t \in (0,T)$,}
	\end{gathered} 
\end{equation}
\begin{equation} \label{eq:measdtc}
\Phi \partial_{t}c \in L^{2}(0,T; (W^{1,s}(\Omega))') \quad \forall s > 2d, 
\end{equation}
\begin{equation} \label{eq:measic}
\Phi c \in C([0, T];(W^{1,s}(\Omega))'), \quad \Phi c(\cdot, 0) = \Phi c_{0} 
\mbox{ in $(W^{1, s}(\Omega))'$} \quad \forall s > 2d,
\end{equation}
\begin{equation} \label{eq:measdgradc}
	\begin{gathered}
	c \mbox{ has a $\{ |\U|>0 \}$-gradient, and}\\
	\D[\circ]{\cdot}{\U}\gradc{|\U|}{0} \in L^{2}(0,T; L^{r}(\Omega)^{d}) \quad \forall r < \frac{2d}{2d-1}, 
	\end{gathered} 
\end{equation}
\begin{equation} \label{eq:measpu}
p \in L^{\infty}(0,T; W^{1, q}_\star(\Omega)), \quad \mathbf{u} \in L^{\infty}(0,T; L^{q}(\Omega)^d) 
\quad \forall q < \frac{d}{d-1}, 
\end{equation}
\begin{equation} \label{eq:measpara}
	\begin{gathered}
	\int_{0}^{T}\langle \Phi \partial_{t}c(\cdot, t), \varphi(\cdot, t)\rangle\ud t \\
	+ \int_{0}^{T}\int_{\Omega} \D[\circ]{x}{\U(x,t)}\gradc{|\U|}{0}(x,t)\cdot\nabla\varphi(x,t)\ud x\ud t\\
	- \int_{0}^{T}\int_{\Omega} c(x,t)\U(x,t)\cdot \nabla \varphi(x, t) \ud x \ud t  \\
	+ \int_{0}^{T}\int_{\Omega} c(x,t)\varphi(x,t)b(x,t)\ud \nu(x) \ud t\\
	=\int_{0}^{T}\int_{\Omega} \hat c(x,t)\varphi(x,t)a(x,t)\ud \nu(x) \ud t
	\quad \forall \varphi \in \bigcup_{s>2d}L^{2}(0,T; W^{1,s}(\Omega)),
	\end{gathered} 
\end{equation}
\begin{equation} \label{eq:measell}
	\begin{gathered}
	\mathbf{u}(x, t) =  -\frac{\mathbf{K}(x)}{\mu(c(x, t))}\nabla p(x,t),  \\
	- \int_{0}^{T}\int_{\Omega}\mathbf{u}(x, t) \cdot \nabla \psi(x, t) \ud x \ud t \\
	= \int_{0}^{T}\int_{\Omega}\left(a-b\right)(x,t)\psi(x,t)\ud \nu(x)\ud t
	\quad \forall \psi \in \bigcup_{q>d} L^{1}(0,T; W^{1,q}(\Omega)).
	\end{gathered} 
\end{equation}
\end{subequations}
\end{theorem}
To reiterate, the duality product in the first term of \eqref{eq:measpara} is between 
$W^{1,s}(\O)$ and its dual.

\begin{remark} Following Remark \ref{rem:freg}, if $c$ is regular then
$\gradc{v}{0}$ can be replaced with $\nabla c$ in \eqref{eq:measpara}.
\end{remark}


\section{Approximate problems and associated estimates} \label{ssec:measapprox}

We obtain the solution $(p,\U,c)$ to \eqref{def:sol.meas} by passing to the limit 
on approximate problems defined below. 
Let $\eps>0$. Replace the molecular diffusion coefficient $d_m$ in \eqref{eq:ddt} with $\eps$ to obtain 
a family of diffusion-dispersion tensors:
\begin{equation}
\D[\eps]{x}{\U} := \Phi(x)\bigg(\eps\mathbf{I}+|\U|\Big(\dl E(\U)+\dt(\mathbf{I}-E(\U))\Big)\bigg). \label{eq:epsddt}
\end{equation}
Then for almost every $x\in\O$, for all $\xi$, $\zeta\in\RR^d$,
\begin{align}
\D[\eps]{x}{\zeta}\xi\cdot\xi &\geq \phi_{\ast}(\varepsilon 
+ \mathrm{min}(\dl,\dt)|\zeta|)|\xi|^{2}, \label{eq:DNcoerc}\\
|\D[\eps]{x}{\zeta}| &\leq \phi_{\ast}^{-1}(\varepsilon + \mathrm{max}(\dl,\dt)|\zeta|).\label{eq:DNbound}
\end{align}
Moreover, writing $\DD_\eps^{1/2}$ for the square-root of $\DD_\eps$ (which is well-defined 
since $\DD_\eps$ is positive-definite), one can show that
\begin{equation}\label{eq:Dsq:bound}
|\Dsq[\eps]{\cdot}{\zeta}|\le \phi_\ast^{-1/2}(\eps+\max(\dl,\dt)|\zeta|)^{1/2}.
\end{equation}
In order to define our approximate problems, we need access to the solution when the
source terms are regular and the molecular diffusion is fixed. To this end, replace
$\DD$ by $\DD_\eps$ in \eqref{eq:model} and fix both
$\eps$ and $\nuapprox\in L^{2}(\O)$ (where $n\in\NN$ will vary in subsequent notions of solution). 
Then Feng \cite{fe95} and Chen and Ewing 
\cite{ce99} show that there exists a weak solution 
$(\prss[n]{\eps},\darcy[n]{\eps},\conc[n]{\eps})$ to \eqref{eq:model} satisfying
\begin{equation}\label{def:sol.L2.N}
\begin{gathered}
\conc[n]{\eps}\in L^2(0,T;H^1(\O)),\quad 0\leq\conc[n]{\eps}(x,t)\leq 1 \quad \mbox{for a.e.\ $(x,t)\in\O\times(0,T)$,}\\
\Phi \partial_{t}\conc[n]{\eps}\in L^{2}(0,T; (W^{1, 4}(\O))'), \\
\Phi\conc[n]{\eps} \in C([0,T];(W^{1,4}(\O))'),\quad\Phi\conc[n]{\eps}(\cdot,0)
=\Phi c_{0}\mbox{ in $(W^{1, 4}(\O))'$}, \\
\D[\eps]{\cdot}{\darcy[n]{\eps}}\nabla\conc[n]{\eps} \in \Leb[d]{2}{4/3},\\
\prss[n]{\eps} \in L^{\infty}(0,T;H^{1}_\star(\O)),\quad\darcy[n]{\eps}\in L^{\infty}(0,T;L^{2}(\O)^d),\\
\mbox{$(\prss[n]{\eps},\darcy[n]{\eps},\conc[n]{\eps})$ satisfies \eqref{eq:measpara} 
for all $\varphi\in L^{2}(0,T; W^{1,4}(\O))$}\\ 
\mbox{with $\DD_\circ$ and $\gradc{|\U|}{0}$ replaced by $\DD_\eps$ and $\nabla\conc[n]{\eps}$, resp.,}\\
\mbox{$(\prss[n]{\eps},\darcy[n]{\eps},\conc[n]{\eps})$ satisfies \eqref{eq:measell} 
for all $\psi\in L^{1}(0,T; H^{1}(\Omega))$.}
\end{gathered}
\end{equation}
Keeping $\DD_\eps$ (with $\eps$ fixed), consider now $\nu\in\cM_+(\O)$. 
Our previous work \cite{dt14} shows that for every $\eps>0$, there exists a solution 
$(\prss{\eps},\darcy{\eps},\conc{\eps})$ to \eqref{eq:model} in the following sense:
\begin{subequations}\label{def:sol.L2N.meas}
\begin{equation}
	\begin{gathered}
	\conc{\eps}\in L^{2}(0,T;H^{1}(\O)),\quad 0\leq\conc{\eps}(x,t)\leq 1 
	\mbox{\quad for a.e.\ $(x,t)\in\O\times(0,T)$,} \\
	\conc{\eps}\in L^{\infty}(0,T; L^{1}(\O,\nu)), \\
	0 \leq\conc{\eps}(x,t)\leq 1, \mbox{\quad for $\nu$-a.e.\ $x\in\O$, for a.e.\ $t\in (0,T)$,}
	\end{gathered} \label{eq:jdktc}
\end{equation}
\begin{equation}
\Phi\partial_{t}\conc{\eps}\in L^{2}(0,T;(W^{1,s}(\O))')\quad\forall s>2d, \label{eq:jdktdtc}
\end{equation}
\begin{equation}
\Phi\conc{\eps}\in C([0, T];(W^{1,s}(\O))'),\quad \Phi \conc{\eps}(\cdot,0)=\Phi c_{0} 
\mbox{ in $(W^{1,s}(\O))'$}\quad\forall s>2d,\label{eq:jdktic}
\end{equation}
\begin{equation} \label{eq:jdktdgradc}
\D[\eps]{\cdot}{\darcy{\eps}}\nabla\conc{\eps}\in L^{2}(0,T; L^{r}(\O)^{d})\quad\forall r<\frac{2d}{2d-1}, 
\end{equation}
\begin{equation}
\prss{\eps}\in L^{\infty}(0,T; W^{1,q}_\star(\O)),\quad\darcy{\eps}\in L^{\infty}(0,T;L^{q}(\O)^d) 
\quad \forall q < \frac{d}{d-1}, \label{eq:jdktpu}
\end{equation}
\begin{equation}\label{eq:jdktpara}
	\begin{gathered} 
	\mbox{$(\prss{\eps},\darcy{\eps},\conc{\eps})$ satisfies \eqref{eq:measpara}}\\
	\mbox{with $\DD_\circ$ and $\gradc{|\U|}{0}$ replaced by $\DD_\eps$ and $\nabla\conc{\eps}$, resp.,}
	\end{gathered}
\end{equation}	
\begin{equation} \label{eq:jdktell}
\mbox{$(\prss{\eps},\darcy{\eps},\conc{\eps})$ satisfies \eqref{eq:measell}.}
\end{equation}
\end{subequations}

\begin{remark}
Standard arguments show that the integral relation in \eqref{eq:jdktell} is equivalent to
\begin{multline}\label{eq:jdktell.pt}
- \int_{\O}\darcy{\eps}(x,t)\cdot \nabla \psi(x) \ud x =\int_{\O}(a-b)(x,t)\psi(x)\ud \nu(x),\\
\mbox{for a.e.\ $t\in (0,T)$, }\forall \psi \in \bigcup_{q>d} W^{1,q}(\Omega).
\end{multline}
\end{remark}

We are now ready to define precisely the approximate problems that we work with in
the subsequent analysis. 
The following two definitions provide the details.
\begin{definition}[Solution-by-truncation to \eqref{def:sol.L2.N}]\label{def:sol.by.truncation}
Assume \eqref{assum:general}. Fix $\nuapprox\in L^2(\O)$, $\eps>0$ and take $k\in\NN$. 
Define the truncated tensor, for $(x,\zeta)\in\Omega\times\RR^d$, by 
\begin{equation}\label{def:truncatedD}
\DD_\eps^k(x,\zeta)=\DD_\eps\left(x,\min(|\zeta|,k)
\frac{\zeta}{\left|\zeta\right|}\right).
\end{equation}
Then a \emph{solution-by-truncation to \eqref{def:sol.L2.N}} is 
a triple $(\prss[n]{\eps},\darcy[n]{\eps},\conc[n]{\eps})$ that satisfies \eqref{def:sol.L2.N}
and such that if $(\Tprss{\eps}{n}{k},\Tdarcy{\eps}{n}{k},\Tconc{\eps}{n}{k})$ is a solution 
to \eqref{def:sol.L2.N} with $\DD_\eps$ replaced by $\DD_\eps^k$
then, along a subsequence as $k\to\infty$,
\begin{equation}\label{conv.Nk}
	\begin{aligned}
	&\Tprss{\eps}{n}{k}\to \prss[n]{\eps}\mbox{ strongly in $L^2(0,T;H^1(\O))$},\\
	&\Tdarcy{\eps}{n}{k}\to \darcy[n]{\eps}\mbox{ strongly in $\Leb[d]{2}{2}$, and}\\
	&\Tconc{\eps}{n}{k}\to \conc[n]{\eps}\mbox{ a.e.\ on $\O\times (0,T)$ and weakly in $L^2(0,T;H^1(\O))$}.
	\end{aligned}
\end{equation}
\end{definition}
\begin{remark} \label{rem:truncD}
Our previous work \cite[Section 3.3]{dt14} establishes the existence of a
solu\-tion-by-truncation to \eqref{def:sol.L2.N}.
The interest in considering $\nuapprox\in L^2(\O)$ and a truncated (and therefore bounded)
diffusion-dispersion tensor is twofold.
It enables us to consider test functions $\varphi\in L^2(0,T;H^1(\O))$ for the 
concentration equation, so that $\varphi=\Tconc{\eps}{n}{k}$ is an admissible test function. 
The concentration equation then shows that $\Phi\partial_t\Tconc{\eps}{n}{k}\in L^2(0,T;(H^1(\O))')$.
\end{remark}
\begin{definition}[Solution-by-approximation to \eqref{def:sol.L2N.meas}]\label{def:sol.by.approx}
Assume \eqref{assum:general}.
A \emph{solution-by-approximation to \eqref{def:sol.L2N.meas}} is a triple $(\prss{\eps},\darcy{\eps},\conc{\eps})$
satisfying \eqref{def:sol.L2N.meas} and such that there exists 
$(\nuapprox)_{n\in\NN}\subset L^2(\O)$, $(\aapprox)_{n\in\NN}\subset 
L^\infty(0,T;C(\ol{\O}))$ and $(\prss[n]{\eps},\darcy[n]{\eps},\conc[n]{\eps})_{n\in\NN}$, with
\begin{itemize}
\item $\nuapprox\geq0$, $\nuapprox\to\nu$ in $(C(\ol{\O}))'\cap (W^{1,\ell}(\O))'$ 
weak-$*$ as $n\to\infty$ (for all $\ell>2$),
and for all $\eta>0$ there exists $N\in\NN$ such that if $n\geq N$,
$\supp(\nuapprox)\subset \supp(\nu)+B(0,\eta)$,
\item $\aapprox\geq 0$, $(\aapprox)_{n\in\NN}$ is bounded in $L^\infty(0,T;C(\ol{\O}))$ and
$\aapprox\to a$ a.e.\ on $\O\times (0,T)$ as $n\to\infty$,
\item $(\nuapprox,\aapprox,b)$ satisfy the compatibility condition \eqref{hyp:compatibility},
\item $(\prss[n]{\eps},\darcy[n]{\eps},\conc[n]{\eps})$ is a solution-by-truncation 
to \eqref{def:sol.L2.N} with $(\nu,a)$ replaced by $(\nuapprox,\aapprox)$,
\end{itemize}
and, along a sequence as $n\to\infty$,
\begin{equation}\label{conv.Neps}
\begin{aligned}
&\prss[n]{\eps}\to\prss{\eps}\mbox{ strongly in $L^2(0,T;W^{1,q}(\O))$ for all $q<\frac{d}{d-1}$},\\
&\darcy[n]{\eps}\to\darcy{\eps}\mbox{ strongly in $\Leb[d]{2}{q}$ for all $q<\frac{d}{d-1}$, and}\\
&\conc[n]{\eps}\to\conc{\eps}\mbox{ a.e.\ on $\O\times (0,T)$ and weakly in $L^2(0,T;H^1(\O))$}.
\end{aligned}
\end{equation}
\end{definition}
\begin{remark}
The existence of a solution-by-approximation to \eqref{def:sol.L2N.meas} is known \cite[Section 4.3]{dt14}. 
Fabrie--Gallou\"et \cite[Section 5]{fg00} establishes the existence of an approximation 
$(\nuapprox,\aapprox)$ of $(\nu,a)$ that satisfies the requirements of Definition
\ref{def:sol.by.approx}. 
\end{remark}
Table \ref{table:solutions} helps to visualise the relationship between these notions
of solution to \eqref{eq:model}. 
Access to the solution $(\Tprss{\eps}{n}{k},\Tdarcy{\eps}{n}{k},\Tconc{\eps}{n}{k})$
of the truncated problem is only required for Lemma \ref{lem:measdarcylimit} and 
the first step of Lemma \ref{lem:strconv.pN}. The rest of the analysis is largely conducted
on the solution-by-approximation $(\prss{\eps},\darcy{\eps},\conc{\eps})$.
\renewcommand{\arraystretch}{1.5}
\renewcommand{\tabcolsep}{0.2cm}
\begin{table}[h!] 
\caption{Notions of solution to \eqref{eq:model}}
\label{table:solutions}
\begin{tabular}[c]{c c}
\hline
Data & Solution \\
\hline
$\DD_\eps^k$, $\eps=\dm>0$, $\nuapprox\in L^2(\O)$ & 
$(\Tprss{\eps}{n}{k},\Tdarcy{\eps}{n}{k},\Tconc{\eps}{n}{k})$\\
\multicolumn{2}{c}{$\bigg\downarrow k\to\infty$}\\
$\eps=\dm>0$, $\nuapprox\in L^2(\O)$ & solution-by-truncation 
$(\prss[n]{\eps},\darcy[n]{\eps},\conc[n]{\eps})$\\
\multicolumn{2}{c}{$\bigg\downarrow n\to\infty$}\\
$\eps=\dm>0$, $\nu\in\cM_+(\O)$ & solution-by-approximation 
$(\prss{\eps},\darcy{\eps},\conc{\eps})$\\
\multicolumn{2}{c}{$\bigg\downarrow \eps\to0$}\\
$\dm=0$, $\nu\in\cM_+(\O)$ & solution $(p,\U,c)$ to \eqref{def:sol.meas}\\
\hline
\end{tabular}
\end{table}

We now recall the estimates necessary for our subsequent analysis. Taking $\Tconc{\eps}{n}{k}$
as a test function in its own equation \cite[Eq. (3.8)]{dt14} (see Remark \ref{rem:truncD})
gives a bound on 
$\DD_\eps^k(\cdot,\Tdarcy{\eps}{n}{k})\nabla\Tconc{\eps}{n}{k}\cdot\nabla\Tconc{\eps}{n}{k}$
in $L^1(\O\times(0,T))$ that is independent of $k$, $n$ and $\eps$. 
Passing to the limit as $k\to\infty$ and then as $n\to\infty$ (in that order) gives
\begin{equation}
	\norm{\Dsq[\eps]{\cdot}{\darcy{\eps}}\nabla\conc{\eps}}_{\Leb[d]{2}{2}}=
	\norm{\D[\eps]{\cdot}{\darcy{\eps}}\nabla\conc{\eps}\cdot\nabla\conc{\eps}}_{L^1(\O\times(0,T))}
	\leq \cter{dmeas}, \label{eq:measenergyest}
\end{equation}
where $\ctel{dmeas}$ does not depend on $\eps$. 
It is well-known \cite{bg89,bg92,fg00} that for all $q\in[1,\frac{d}{d-1})$ there exists a 
constant $\ctel{measellest}$ not depending on $\eps$ such that
\begin{equation} \label{eq:measdarcyest}
\norm{\prss{\eps}}_{L^{\infty}(0,T; W^{1,q}(\O))} \leq \cter{measellest}\mbox{ and }
\norm{\darcy{\eps}}_{\Leb[d]{\infty}{q}} \leq \cter{measellest}.
\end{equation}
Estimates \eqref{eq:Dsq:bound} and \eqref{eq:measdarcyest} give a bound on
$|\Dsq[\eps]{\cdot}{\darcy{\eps}}|$ in $L^\infty(0,T;L^s(\O))$ for all $s<2d/(d-1)$.
Combined with \eqref{eq:measenergyest}, the decomposition 
$\D[\eps]{\cdot}{\darcy{\eps}}\nabla\conc{\eps} = 
\Dsq[\eps]{\cdot}{\darcy{\eps}}\Dsq[\eps]{\cdot}{\darcy{\eps}}\nabla\conc{\eps}$
and H\"older's inequality,
this shows that for all $r<\frac{2d}{2d-1}$, 
there exists a constant $\ctel{measd.r}$ not depending on $\eps$ such that
\begin{equation} \label{eq:measdest}
\norm{\D[\eps]{\cdot}{\darcy{\eps}}\nabla\conc{\eps}}_{\Leb[d]{2}{r}} \leq \cter{measd.r}.
\end{equation}
Applying the coercivity \eqref{eq:DNcoerc} to \eqref{eq:measenergyest} gives
\begin{equation}	
\norm{|\darcy{\eps}|^{1/2}\nabla\conc{\eps}}_{\Leb[d]{2}{2}}
\leq \cter{dmeas}^{1/2}\phi_{\ast}^{-1/2}\min(\dl,\dt)^{-1/2}. \label{eq:measuhalfgradcest}
\end{equation} 
As for \eqref{eq:measdest}, from estimates \eqref{eq:measdarcyest} and \eqref{eq:measuhalfgradcest},
for every $r\in[1, \frac{2d}{2d-1})$ we obtain the existence of a constant 
$\ctel{measugradc}$ not depending on $\eps$ such that
\begin{equation} \label{eq:measugradcest}
\norm{|\darcy{\eps}|\nabla\conc{\eps}}_{\Leb[d]{2}{r}} \leq \cter{measugradc}.
\end{equation}
Finally, from \eqref{eq:jdktpara} and the previous estimates, for every $s>2d$ there is
a constant $\ctel{measdtc1}$ not depending $\eps$ such that
\begin{equation} \label{eq:measdtcest}
\norm{\Phi\partial_{t}\conc{\eps}}_{L^{2}(0,T;(W^{1,s}(\Omega))')}\leq\cter{measdtc1}.
\end{equation}
\begin{remark} In \cite[Corollary 3.2]{dt14}, it is stated that the constant $\cter{measd.r}$ depends
only on $\norm{\nu}_{L^1(\Omega)}$, when in fact it depends on $\norm{\nu}_{L^2(\O)}$. 
This is a typographical error and has no impact on the other results in that article.
\end{remark}
\begin{remark}\label{rem:W1l} 
Using the regularity result of Monier and Gallou\"et 
\cite{mg99} and the fact that $\nu\in (W^{1,\ell}(\O))'$ for all $\ell>2$, as in 
Fabrie--Gallou\"et \cite{fg00} we see that \eqref{eq:measdarcyest}
holds for any $q<2$. In order to demonstrate that this additional regularity is required in only a few places,
we retain \eqref{eq:measdarcyest} and all subsequent estimates with $q<d/(d-1)$.

By using the Stampacchia formulation of the solution to linear elliptic equations with measures
\cite{stam65}, we previously analysed \cite{dt14}
the model \eqref{eq:model} for $d_m>0$. This Stampacchia formulation
provides the uniqueness of the solution to linear elliptic equations with measure data,
without the additional $(W^{1,\ell}(\O))'$ regularity
assumption. However, it is unclear if our reasoning below could be adapted to
this formulation, rather than the (more natural) \eqref{sol:limit.ell}.
\end{remark}


\section{Proof of Theorem \ref{th:main}} \label{sec:measprob}


\subsection{Improving local elliptic regularity} \label{ssec:measest}   

Multiplying the elliptic equation \eqref{eq:darcy} by an appropriately chosen cutoff function $\theta$
excises the singularities caused by the measure sources and localises the problem
to regions where the absolute permeability $\K$ is regular. Our next lemma follows
an analogous procedure to Amirat--Ziani by rewriting the pressure equation in a form
that yields higher local regularity of the solution.
\begin{lemma} \label{lem:measgrisvardest}
Assume \eqref{assum:general}.
For $\eps>0$, let $(\prss{\eps},\darcy{\eps},\conc{\eps})$ be a solution-by-approxi\-mation 
to \eqref{def:sol.L2N.meas}. Let $\theta\in C^\infty_c(\O)$ be such that 
$\supp(\theta)\cap(\discK\cup\supp(\nu))=\emptyset$
and take $r<\frac{2d}{2d-1}$.
Then there exists $\ctel{measgris1}$, depending on $\theta$ and $r$ but not on $\eps$, such that
\begin{align}
\norm{\theta\prss{\eps}}_{L^{\infty}(0,T;H^1(\O))}&\leq \cter{measgris1}, \label{eq:measgrisest0}\\
\norm{\theta\prss{\eps}}_{L^{2}(0,T;W^{2,r}(\O))}&\leq \cter{measgris1}, \label{eq:measgrisest1}\\
\norm{\theta\darcy{\eps}}_{L^{2}(0,T;W^{1,r}(\O)^{d})}&\leq \cter{measgris1}. \label{eq:measgrisest2}
\end{align}
\end{lemma}
\begin{proof}
\emph{Step 1: proof of \eqref{eq:measgrisest0}.}

Consider \eqref{eq:jdktell.pt} with $a$ and $\nu$ replaced by $\aapprox$ and $\nuapprox$.
For almost every $t\in(0,T)$, the local elliptic estimates in Droniou \cite[Theorem 2]{D03} 
show that $\prss[n]{\eps}$ satisfies \eqref{eq:measgrisest0} with a bound not depending 
on $n$ or $\eps$. Passing to the limit as $n\to\infty$ shows that \eqref{eq:measgrisest0} holds.

\emph{Step 2: derivation of localised equation.}

Take $\theta$ satisfying the hypotheses of the lemma, and consider \eqref{eq:darcy} 
with $p$, $\U$ and $c$ replaced by $\prss{\eps}$, $\darcy{\eps}$ and $\conc{\eps}$, 
respectively. Multiplying the first equation by $\theta$ gives, in the sense of 
distributions,
\begin{equation*}
\theta\darcy{\eps} = -\theta\frac{\K}{\mu(\conc{\eps})}\nabla\prss{\eps} 
= -\frac{\K}{\mu(\conc{\eps})}\nabla (\theta\prss{\eps}) 
+ \prss{\eps}\frac{\K}{\mu(\conc{\eps})}\nabla\theta.
\end{equation*} 
Multiplying the second equation by $\theta$ yields
\begin{equation*}
\dive(\theta\darcy{\eps})-\darcy{\eps}\cdot\nabla\theta = \theta(a-b)\nu.
\end{equation*}
The property of $\supp(\nu)$ and the choice of $\theta$ show that $\theta\nu=0$, 
so the right-hand side of the previous equality vanishes.
Combining these expressions using standard computations that are justified (in the sense of distributions)
by the regularity \eqref{eq:measgrisest0}, then simplifying
where appropriate using the definition of $\darcy{\eps}$ leads to
\begin{equation} \label{eq:thetap3}
	-\dive{(\K\nabla(\theta\prss{\eps}))} = -\prss{\eps}\dive{(\K\nabla\theta)} 
	+ \theta\mu'(\conc{\eps})\darcy{\eps}\cdot\nabla\conc{\eps}
	+ 2\mu(\conc{\eps})\darcy{\eps}\cdot\nabla\theta.
\end{equation}
In order to apply Grisvard's estimates we require that the diffusion matrix belong to the 
class $C^{0,1}(\ol{\O}; \mathcal{S}_{d}(\RR))$. Note that each term in 
\eqref{eq:thetap3} contains $\theta$, so that both sides vanish outside the 
support of $\theta$. We may therefore replace $\K$ in \eqref{eq:thetap3} by a 
uniformly coercive Lipschitz tensor $\tK$ that agrees with $\K$ on $\supp(\theta)$ 
whilst retaining equality.
Take $\rho\in C^{\infty}_{c}(\O)$ with $0 \leq \rho \leq 1$ and such that
$\supp(\theta)\subset \omega \subset \supp(\rho) \subset \O\setminus \discK$,
where $\omega$ is an open set such that $\rho \equiv 1$ on $\omega$. Define
\begin{equation}\label{def:tK}
\tK := \rho\K + (1-\rho)\mathbf{I}.
\end{equation}
Then $\tK\in C^{0,1}(\ol{\O}; \mathcal{S}_{d}(\RR))$ and satisfies 
$\tK = \K$ on $\supp(\theta)$, $\tK = \mathbf{I}$ outside $\supp(\rho)$. 
Furthermore, for almost every $x\in\O$ and for all $\xi\in\RR^d$ we have
$\tK(x)\xi\cdot\xi \geq \mathrm{min}(1, k_{\ast})|\xi|^2$.
Replacing $\K$ with $\tK$ in the first two terms of \eqref{eq:thetap3}, we are lead
to the following localised pressure equation:
\begin{equation} \label{eq:measgris1} 
	-\dive{\left(\tK\nabla(\theta \prss{\eps})\right)} 
	=-\prss{\eps}\dive{\left(\tK\nabla\theta\right)}
	+ \theta\mu'(\conc{\eps})\darcy{\eps}\cdot\nabla \conc{\eps}
	+ 2\mu(\conc{\eps})\darcy{\eps}\cdot\nabla\theta.
\end{equation}

\emph{Step 3: proof of local estimates.}

The equation \eqref{eq:measgris1} is satisfied on $\O\times (0,T)$, 
but due to the compact support of $\theta$, it also holds on $B\times (0,T)$, 
where $B$ is ball containing $\O$ and all functions are extended by $0$ outside
$\O$. Estimates \eqref{eq:measdarcyest} and \eqref{eq:measugradcest} then show that 
for every $1\leq r<2d/(2d-1)$ the right-hand side of \eqref{eq:measgris1} 
is bounded in $\Leb{2}{r}$ uniformly in $\eps$. 
Then \eqref{eq:measgrisest1} follows from Grisvard \cite[Eq. (2.3.3.1)]{gr85} 
and from the fact that $\theta \prss{\eps}(\cdot,t)\in H^1_0(B)$ for almost every $t\in (0,T)$.

For \eqref{eq:measgrisest2}, write $\tK = (\tilde{k}_{ij})_{i,j=1,\ldots,d}$. 
Observing the summation convention, the $i$-th component of $\theta\darcy{\eps}$ is
\begin{equation*}
(\theta\darcy{\eps})_i = -\frac{\theta}{\mu(\conc{\eps})}\tilde{k}_{ij}\partial_{x_{j}}\prss{\eps}.
\end{equation*} 
By the regularity properties \eqref{eq:jdktc} and \eqref{eq:jdktpu} of $\prss{\eps}$ and $\conc{\eps}$,
we can write, in the sense of distributions,
\begin{align*}
\partial_{x_{l}}(\theta\darcy{\eps})_i ={}& 
(\partial_{x_{l}}\theta)\left( -\frac{\tilde{k}_{ij}}{\mu(\conc{\eps})}\partial_{x_{j}}\prss{\eps}\right) 
- \theta\frac{\mu'(\conc{\eps})}{\mu(\conc{\eps})}(\darcy{\eps})_i\partial_{x_{l}}\conc{\eps} \nonumber\\
&- \frac{\theta}{\mu(\conc{\eps})}(\partial_{x_{l}}\tilde{k}_{ij})
(\partial_{x_{j}}\prss{\eps})-\frac{\theta}{\mu(\conc{\eps})}\tilde{k}_{ij}
\partial_{x_{l}x_{j}}^{2}\prss{\eps}\nonumber\\
={}& T_1 + T_2 + T_3 + T_4. 
\end{align*}
Thanks to \eqref{eq:measdarcyest}, both $T_1$ and $T_3$ are bounded in $\Leb{\infty}{q}$ 
for every $q<\frac{d}{d-1}$. For every $r<\frac{2d}{2d-1}$, estimate term $T_2$ 
in $\Leb{2}{r}$ using \eqref{eq:measugradcest}. For $T_4$, use \eqref{eq:measgrisest1}.
\end{proof}


\subsection{Extraction of converging sequences} \label{ssec:measlimit}

From \eqref{eq:jdktc} the sequence $(\conc{\eps})_{\eps>0}$ is bounded in $L^\infty(\O\times(0,T))$, 
so that up to a subsequence
\begin{equation} \label{eq:meascntoc}
\conc{\eps}\weakto c \quad \mbox{in $L^\infty(\O\times(0,T))$ weak-$\ast$},
\quad 0 \leq c \leq 1 \mbox{ a.e.\ in }\O\times(0, T),
\end{equation}
which proves the first part of \eqref{eq:measc}; the second part follows at the end
of Section \ref{ssec:passlimitconc}. Estimate \eqref{eq:measdarcyest} implies the existence 
of extracted subsequences such that
\begin{align}
\prss{\eps} &\weakto p \quad \mbox{in $L^{\infty}(0,T; W^{1,q}_\star(\O))$ weak-$\ast$} 
\quad \forall 1\leq q < \frac{d}{d-1}, \mbox{ and} \label{eq:measpntop}\\
\darcy{\eps} &\weakto \U \quad \mbox{in $\Leb[d]{\infty}{q}$ weak-$\ast$}\quad \forall 1\leq q < \frac{d}{d-1},
\label{eq:measuntou} 
\end{align}
which proves \eqref{eq:measpu}. The porosity is independent of time, so for every $s>2d$,
\eqref{eq:measdtcest} provides an estimate in $L^{2}(0,T; (W^{1,s}(\O))')$ of the 
sequence $(\partial_{t}(\Phi\conc{\eps}))_{\eps>0}$, from which we conclude that
\begin{equation}\label{eq:lim.dt}
\Phi\partial_{t}\conc{\eps}\weakto\Phi\partial_{t}c\quad\mbox{ weakly in $L^{2}(0,T;(W^{1,s}(\O))')$
for all $s>2d$},
\end{equation}
thus proving \eqref{eq:measdtc}. Furthermore, $(\Phi\conc{\eps})_{\eps>0}$ is bounded 
in $L^\infty(\O\times(0,T))$, and $L^\infty(\O)$ is compactly embedded in $(W^{1,s}(\O))'$
(since $W^{1,s}(\O)$ is compactly and densely embedded in $L^1(\O)$).
A classical compactness lemma due to Simon \cite{sim-87-com} therefore
ensures that, up to a subsequence,
$\Phi\conc{\eps}\to\Phi c$ in $C([0,T]; (W^{1,s}(\O))')$ for all $s>2d$,
which proves \eqref{eq:measic}. 

\subsection{Passing to the limit in the pressure equation}

The proof that $(p,\U,c)$ satisfies the elliptic equation \eqref{eq:measell} will be complete
by passing to the limit in \eqref{eq:jdktell}, provided that we identify $\U$.
For this we follow the ideas of Amirat--Ziani \cite[Lemma 2.4]{az04}, who rely
on a variant of the compensated compactness phenomenon due to Kazhikhov \cite{kaz98}. 
Our proof necessarily departs from that of Amirat--Ziani due to our use of the 
cutoff functions $\theta$. We also correct an error in their estimate of the term
corresponding to our $\partial_{t}\mu(\conc{\eps})$. They claim this sequence is 
bounded $L^2$ in time, when in fact it is only $L^1$ (for both regular and measure
source terms). This necessitates our use of the $BV(0,T)$ spaces and a
compensated compactness result adapted to this regularity, Lemma \ref{lem:cc} in the appendix.

\begin{lemma} \label{lem:measdarcylimit}
Assume \eqref{assum:general} and for $\eps>0$,
let $(\prss{\eps},\darcy{\eps},\conc{\eps})$ be a solution-by-approxi\-ma\-tion to \eqref{def:sol.L2N.meas}.
Assume that \eqref{eq:meascntoc}--\eqref{eq:measuntou} hold. 
Then for almost every $(x,t)\in\O\times(0,T)$,
\begin{equation}
\U(x,t) = -\frac{\K(x)}{\mu(c(x,t))}\nabla p(x,t).\label{eq:measdarcylimit}
\end{equation}
\end{lemma}

\begin{proof}
By the assumptions \eqref{hyp:viscosity} on $\mu$, there exists $\ol{\mu}$, 
$\underline{\mu}\in L^\infty(\O\times(0,T))$ such that $\mu_\ast\le \underline{\mu},\ol{\mu}
\le \mu^\ast$ and, up to a subsequence, 
\begin{equation}\label{conv.cN.meas}
\mu(\conc{\eps}) \weakto \ol{\mu} \quad\mbox{ and }\quad\frac{1}{\mu(\conc{\eps})} 
\weakto \frac{1}{\underline{\mu}} \quad \mbox{ in $L^\infty(\O\times(0,T))$ weak-$\ast$}.
\end{equation}
\emph{Step 1: $BV$ estimates.}

Take $\psi\in C^\infty_c(\O)$. To apply Lemma \ref{lem:cc}, we must estimate the sequences
$\left(\int_\O \Phi\mu(\conc{\eps}(x,\cdot))\psi(x)\ud x\right)_{\eps>0}$ and
$\left(\int_\O \frac{\Phi}{\mu(\conc{\eps}(x,\cdot))}\psi(x)\ud x\right)_{\eps>0}$ 
in the space $BV(0,T)$. We first obtain these estimates on the solution
to the truncated problem from Definition \ref{def:sol.by.truncation}, and then
deduce the corresponding estimates on $(\prss{\eps},\darcy{\eps},\conc{\eps})$.

Replace $\nu$ and $a$ by $\nuapprox$ and $\aapprox$ from Definition \ref{def:sol.by.approx}.
Let $(\Tprss{\eps}{n}{k},\Tdarcy{\eps}{n}{k},\Tconc{\eps}{n}{k})$ be the solution to the corresponding truncated
problem, that is with $\DD_\eps$ replaced by $\DD_\eps^k$, defined by \eqref{def:truncatedD}.

Take $\gamma\in C^{2}([0,1])$, $\psi\in C^{\infty}(\ol{\O})$ 
and choose $\varphi=\gamma'(\Tconc{\eps}{n}{k}))\psi$
as a test function in \cite[Eq. (3.8)]{dt14}. Then for almost every $t\in(0,T)$ we have
\begin{equation}\label{eq:dtmuc}
\begin{gathered}
\langle (\Phi\partial_{t}\Tconc{\eps}{n}{k})(\cdot,t),\gamma'(\Tconc{\eps}{n}{k})(\cdot,t)\psi\rangle_{(H^{1})',H^{1}}\\
+ \int_{\O}\DD_\eps^k(x,\Tdarcy{\eps}{n}{k}(x,t))\nabla\Tconc{\eps}{n}{k}(x,t)
\cdot\nabla[\gamma'(\Tconc{\eps}{n}{k}(x,t))\psi(x)]\ud x \\
- \int_{\O}\Tconc{\eps}{n}{k}(x,t)\Tdarcy{\eps}{n}{k}(x,t)\cdot\nabla[\gamma'(\Tconc{\eps}{n}{k}(x,t))\psi(x)]\ud x\\
+ \int_{\O}\Tconc{\eps}{n}{k}(x,t)\gamma'(\Tconc{\eps}{n}{k}(x,t))\psi(x)b(x,t)\nuapprox(x)\ud x \\
= \int_{\O}\hat{c}(x,t)\gamma'(\Tconc{\eps}{n}{k}(x,t))\psi(x)\aapprox(x,t)\nuapprox(x)\ud x. 
\end{gathered}
\end{equation}
Since $\Phi\partial_{t}\Tconc{\eps}{n}{k}\in L^{2}(0,T;(H^{1}(\O))')$ and 
$\gamma'(\Tconc{\eps}{n}{k})\in L^{2}(0,T;H^{1}(\O))$,
the product $\Phi\partial_{t}\Tconc{\eps}{n}{k}\gamma'(\Tconc{\eps}{n}{k})$ is well-defined as an element of 
$L^{1}(0,T;(C^{\infty}(\ol{\O}))')$. Reasoning by density of smooth functions, we also
see that 
\begin{equation}\label{eq:dtmuc.dt}
\partial_{t}(\Phi\gamma(\Tconc{\eps}{n}{k}))
=\Phi\gamma'(\Tconc{\eps}{n}{k})\partial_{t}\Tconc{\eps}{n}{k}
\quad\mbox{ in $L^{1}(0,T;(C^{\infty}(\ol{\O}))')$}.
\end{equation}
Introducing $\zeta(s)=\int_0^s q\gamma''(q)\ud q$, write
\begin{align*}
\Tconc{\eps}{n}{k}\Tdarcy{\eps}{n}{k}\cdot\nabla\left[\gamma'(\Tconc{\eps}{n}{k})\psi\right]={}&
\left[\Tdarcy{\eps}{n}{k}\cdot\nabla\Tconc{\eps}{n}{k}\right]
\Tconc{\eps}{n}{k}\gamma''(\Tconc{\eps}{n}{k})\psi
+\left[\Tdarcy{\eps}{n}{k}\cdot\nabla \psi\right]\Tconc{\eps}{n}{k}\gamma'(\Tconc{\eps}{n}{k})\\
={}& \Tdarcy{\eps}{n}{k}\cdot\nabla(\zeta(\Tconc{\eps}{n}{k})\psi)
+\left[\Tdarcy{\eps}{n}{k}\cdot\nabla\psi\right]\left(\Tconc{\eps}{n}{k}\gamma'(\Tconc{\eps}{n}{k}) 
-\zeta(\Tconc{\eps}{n}{k})\right).
\end{align*}
The equation \eqref{eq:jdktell.pt} on $\Tdarcy{\eps}{n}{k}$ then shows that
\begin{multline*}
-\int_{\O}\Tconc{\eps}{n}{k}(x,t)\Tdarcy{\eps}{n}{k}(x,t)\cdot
\nabla[\gamma'(\Tconc{\eps}{n}{k}(x,t))\psi(x)]\ud x\\
=\int_{\O}\zeta(\Tconc{\eps}{n}{k}(x,t))(\aapprox-b)(x,t)\psi(x)\nuapprox(x)\ud x \\
-\int_{\O}\left[\Tdarcy{\eps}{n}{k}(x,t)\cdot\nabla\psi(x)\right]
\left(\Tconc{\eps}{n}{k}(x,t)\gamma'(\Tconc{\eps}{n}{k}(x,t))-\zeta(\Tconc{\eps}{n}{k}(x,t))\right)\ud x.
\end{multline*}
Substituted alongside \eqref{eq:dtmuc.dt} in \eqref{eq:dtmuc}, this gives
\begin{equation} \label{eq:dtmuc2}
	\begin{gathered}
	\langle \partial_{t}(\Phi\gamma(\Tconc{\eps}{n}{k})(\cdot, t)), 
	\psi\rangle_{(C^{\infty}(\ol{\O}))',C^{\infty}(\ol{\O})} \\
	+ \int_{\O}\gamma'(\Tconc{\eps}{n}{k}(x,t))\DD_\eps^k(x,\Tdarcy{\eps}{n}{k}(x,t))
	\nabla\Tconc{\eps}{n}{k}(x,t)\cdot\nabla\psi(x)\ud x \\
	+ \int_{\O}\left[\DD_\eps^k(x,\Tdarcy{\eps}{n}{k}(x,t))\nabla\Tconc{\eps}{n}{k}(x,t)
	\cdot\nabla\Tconc{\eps}{n}{k}(x,t)\right]\gamma''(\Tconc{\eps}{n}{k}(x,t))\psi(x)\ud x \\
	+\int_{\O}\zeta(\Tconc{\eps}{n}{k}(x,t))(\aapprox-b)(x,t)\psi(x)\nuapprox(x)\ud x\\
	-\int_{\O}\left[\Tdarcy{\eps}{n}{k}(x,t)\cdot\nabla\psi(x)\right]
	\left(\Tconc{\eps}{n}{k}(x,t)\gamma'(\Tconc{\eps}{n}{k}(x,t))
	-\zeta(\Tconc{\eps}{n}{k}(x,t))\right)\ud x\\
	+ \int_{\O}\Tconc{\eps}{n}{k}(x,t)\gamma'(\Tconc{\eps}{n}{k}(x,t))\psi(x)b(x,t)\nuapprox(x)\ud x\\
	= \int_{\O}\hat{c}(x,t)\gamma'(\Tconc{\eps}{n}{k}(x,t))\psi(x)\aapprox(x,t)\nuapprox(x)\ud x. 
	\end{gathered}
\end{equation}
All the integral terms can be bounded in the $L^1(0,T)$ norm by
using $0\leq\Tconc{\eps}{n}{k}\leq1$ and the estimates \eqref{eq:measenergyest}, 
\eqref{eq:measdarcyest} and \eqref{eq:measdest} for 
$(\Tprss{\eps}{n}{k},\Tdarcy{\eps}{n}{k},\Tconc{\eps}{n}{k})$ (and $\DD_\eps$
replaced by $\DD_\eps^k$), with constants that do not depend
on $k$, $n$ or $\eps$. This gives
\[
\seminorm{\int_\O (\Phi\gamma(\Tconc{\eps}{n}{k}))(x,\cdot)\psi(x)\ud x}_{BV(0,T)}
=\norm{\partial_{t}\int_\O (\Phi\gamma(\Tconc{\eps}{n}{k}))(x,\cdot)\psi(x)\ud x}_{L^1(0,T)}\\
\leq\cter{dtmuc1},
\]
where $\ctel{dtmuc1}$ may depend on $\psi$ and $\gamma$, but not on $k$, $n$ or $\eps$. 
Letting $k\to\infty$, $\Tconc{\eps}{n}{k}\to\conc[n]{\eps}$ almost-everywhere and so
$\seminorm{\int_\O \Phi\gamma(\conc[n]{\eps})(x,\cdot)\psi(x)\ud x}_{BV(0,T)}\leq\cter{dtmuc1}$.
By the convergence \eqref{conv.Neps} of $\conc[n]{\eps}$ to $\conc{\eps}$, we infer a
uniform-in-$\eps$ estimate in $BV(0,T)$ of $\int_\O \Phi \gamma(\conc{\eps})(x,\cdot)\psi(x)\ud x$.
Finally, set $\gamma=\mu$ or $\frac{1}{\mu}$ to see that
\begin{equation}\label{bound.BV}
\begin{gathered}
\left(\int_\O \Phi(x)\mu(\conc{\eps}(x,\cdot))\psi(x)\ud x\right)_{\eps>0}\mbox{ and }
\left(\int_\O \frac{\Phi(x)}{\mu(\conc{\eps}(x,\cdot))}\psi(x)\ud x\right)_{\eps>0}\\
\mbox{are bounded in $BV(0,T)$.}
\end{gathered}
\end{equation}

\emph{Step 2: passing to the limit on $\mu(\conc{\eps})\darcy{\eps}$.}

For $q\in[1,\frac{d}{d-1})$, the sequence $(\mu(\conc{\eps})\darcy{\eps})_{\eps>0}$ is 
bounded in $\Leb[d]{\infty}{q}$, so there exists $\ol{\mu\U} \in \Leb[d]{\infty}{q}$ 
such that, up to a subsequence,
\begin{equation}
\mu(\conc{\eps})\darcy{\eps}\weakto \ol{\mu\U} \quad\mbox{in $\Leb[d]{\infty}{q}$ 
weak-$\ast$ for all $q<\frac{d}{d-1}$}.
\label{eq:darcylim1.meas}
\end{equation}
The estimates \eqref{eq:measgrisest2} and \eqref{bound.BV} and the weak convergences 
\eqref{eq:measuntou} and \eqref{conv.cN.meas} enable us to apply Lemma \ref{lem:cc} 
with $p=2$, $a=r$ (for a fixed $r<2d/(2d-1)$), $\alpha_\eps =$ components of $\theta\darcy{\eps}$ 
and $\beta_\eps = \Phi\mu(\conc{\eps})$, to see that
\[
\theta\Phi\mu(\conc{\eps})\darcy{\eps} \weakto \theta\Phi\ol{\mu}\U \quad\mbox{in }\distt.
\]
Combined with \eqref{eq:darcylim1.meas} multiplied by $\theta\Phi$, this shows that
$\theta\Phi\ol{\mu\U}=\theta\Phi\ol{\mu}\U$ almost-everywhere.
This holds for any $\theta\in C^{\infty}_{c}(\O)$ with 
$\supp(\theta)\cap \discK = \emptyset$. By the freedom of $\theta$ and since $\Phi$ is uniformly
positive, so we conclude that $\ol{\mu\U}=\ol{\mu}\U$ almost-everywhere and hence
\begin{equation*}
\mu(\conc{\eps})\darcy{\eps} \weakto \ol{\mu}\U \quad\mbox{in $\Leb[d]{\infty}{q}$ weak-$\ast$ 
for all $q<\frac{d}{d-1}$}.  
\end{equation*}
Note that by \eqref{eq:measpntop}, $\mu(\conc{\eps})\darcy{\eps}=-\K\nabla\prss{\eps}\weakto -\K\nabla p$ in 
$\Leb[d]{\infty}{q}$ weak-$\ast$ for all $q<d/(d-1)$.
Thus for almost every $(x,t)\in\O\times(0,T)$,
\begin{equation}
\U (x,t)= -\frac{\K(x)}{\ol{\mu}(x,t)}\nabla p(x,t). \label{eq:uid1.meas}
\end{equation}
\emph{Step 3: identifying the limit of $\darcy{\eps}$.}

We seek to identify the limit of
\begin{equation} \label{eq:ccsplit}
\Phi\theta\darcy{\eps} = -\Phi\theta\frac{\tK}{\mu(\conc{\eps})}\nabla\prss{\eps}.
\end{equation}
Apply Lemma \ref{lem:cc} to the right-hand side, with 
$p=2$, $a=r$ (for a fixed $r<2d/(2d-1)$), $\alpha_\eps=$ components of $-\theta\tK\nabla\prss{\eps}$
and $\beta_\eps=\frac{\Phi}{\mu(\conc{\eps})}$. The estimates \eqref{eq:measgrisest1} and
\eqref{bound.BV} and the convergences \eqref{eq:measpntop} and \eqref{conv.cN.meas}
once again show that the assumptions of Lemma \ref{lem:cc} are satisfied.
We then pass to the limit on both sides of \eqref{eq:ccsplit} to obtain 
\[
\Phi\theta\U 
= -\Phi\theta\frac{\tK}{\underline{\mu}}\nabla p=-\Phi\theta\frac{\K}{\underline{\mu}}\nabla p.
\]
Again using the freedom of $\theta$ and the strict positivity of $\Phi$, for
almost every $(x,t)\in\O\times(0,T)$,
\begin{equation}
\U(x,t) = -\frac{\K(x)}{\underline{\mu}(x,t)}\nabla p(x,t). \label{eq:uid2.meas}
\end{equation}
Comparing \eqref{eq:uid1.meas} and \eqref{eq:uid2.meas}, for almost every $(x,t)\in\O\times(0,T)$,
\begin{equation*}
(\ol{\mu}\U)(x,t) = (\underline{\mu}\U)(x,t).
\end{equation*}
To conclude the proof of \eqref{eq:measdarcylimit}, argue exactly as in Amirat--Ziani \cite[Lemma 2.4]{az04}.
\end{proof}

\subsection{Strong convergence of the Darcy velocity} \label{ssec:strongconvU}
The strong convergence of the Darcy velocity is necessary to handle the convergence
of the diffusion-dispersion term, detailed in Section \ref{ssec:passlimitconc}. Strong
convergence of $(\darcy{\eps})_{\eps>0}$ begins with strong convergence of 
$(\nabla\prss{\eps})_{\eps>0}$.
When the source terms belong to $\Leb{\infty}{2}$, the key to proving the latter
is to use $\prss{\eps}-p$ as a test function in its own equation (see \cite[Lemma 2.5]{az04}). 
In the non-variational setting of measure source terms, this is no longer possible 
as $\prss{\eps}-p$ does not have the required regularity. We first need to excise the 
support of the measure using localisation functions $\theta$. While doing so, we 
create lower order terms in the right-hand side whose convergence needs to be assessed. 
This is the purpose of the following lemma, which establishes the strong convergence 
of $(\prss{\eps})_{\eps>0}$. Due to the lack of estimates on the time derivative of 
$(\prss{\eps})_{\eps>0}$, this result is not straightforward and requires careful use of 
the Aubin--Simon compactness lemma, alongside a uniqueness result for elliptic 
equations with source terms in $\cM_{+}(\O)\cap (W^{1,\ell}(\O))'$ for all $\ell>2$.

\begin{lemma}\label{lem:strconv.pN}
Assume \eqref{assum:general}. For $\eps>0$,
let $(\prss{\eps},\darcy{\eps},\conc{\eps})$ be a solution-by-approxi\-ma\-tion to \eqref{def:sol.L2N.meas}.
Assume that \eqref{eq:meascntoc}--\eqref{eq:measuntou} hold along
a subsequence. Then along the same subsequence,
\begin{equation}\label{eq:str.conv.pN.Lq}
\prss{\eps}\to p\mbox{ strongly in $\Leb{a}{q}$ for all $a<\infty$ and all $q<\frac{d}{d-1}$,}
\end{equation}
and for any $\theta\in C^\infty_c(\O)$ such that $\supp(\theta)\cap (\discK\cup\supp(\nu))=\emptyset$,
\begin{equation}\label{eq:str.conv.pN.L2}
\theta\prss{\eps}\to\theta p\mbox{ strongly in $\Leb{a}{2}$ for all $a<\infty$.}
\end{equation}
\end{lemma}

\begin{proof}
\emph{Step 1: almost-everywhere convergence of $1/\mu(\conc{\eps})$.}

Our aim is to apply an Aubin--Simon lemma to $1/\mu(\conc{\eps})$. We can only estimate
the time derivative of this function when multiplied by the porosity $\Phi$. 
To eliminate this factor, we use a similar trick as in our previous work \cite[Section 3.3]{dt14}.
Let $\delta\in (1,\infty)$ and set 
$\frac{1}{\Phi}W^{1,\delta}(\O)=\{v\in L^\delta(\O)\,:\,\Phi v\in W^{1,\delta}(\O)\}$,
with norm $\norm{v}_{\frac{1}{\Phi}W^{1,\delta}(\O)}=\norm{\Phi v}_{W^{1,\delta}(\O)}$.
By the Rellich theorem, $\frac{1}{\Phi}W^{1,\delta}(\O)$ is compactly and densely embedded in
$L^\delta(\O)$. It follows that $L^{\delta'}(\O)$ is compactly embedded in 
$(\frac{1}{\Phi}W^{1,\delta}(\O))'$.

Take $(\Tprss{\eps}{n}{k},\Tdarcy{\eps}{n}{k},\Tconc{\eps}{n}{k})$ as in the proof 
of Lemma \ref{lem:measdarcylimit} and fix $s>2d$. 
The family $(\Tprss{\eps}{n}{k},\Tdarcy{\eps}{n}{k},\Tconc{\eps}{n}{k}))_{\eps>0}^{n,k\in\NN}$ 
satisfies estimates \eqref{eq:measenergyest}--\eqref{eq:measdest},
with constants not depending on $n$, $k$ or $\eps$. Used in \eqref{eq:dtmuc2} 
applied to $\gamma=1/\mu$, these estimates give
a uniform bound on $\partial_t (\Phi/\mu(\Tconc{\eps}{n}{k}))$ in $L^1(0,T;(W^{1,s}(\O))')$
and therefore in $L^1(0,T;(W^{1,s}(\O))'+(\frac{1}{\Phi}W^{1,\delta}(\O))')$.

Now $(\Phi/\mu(\Tconc{\eps}{n}{k}))_{\eps>0}^{n,k\in\NN}$ is bounded 
in $L^\infty(\O\times(0,T))$ and therefore in $L^1(0,T;L^{\delta'}(\O))$, with $L^{\delta'}(\O)$ compactly
embedded in $(\frac{1}{\Phi}W^{1,\delta}(\O))'$.
Classical Aubin--Simon lemmas show that $A=\{\Phi/\mu(\Tconc{\eps}{n}{k})\,:\,\eps>0;\,n,k\in\NN\}$
is relatively compact in the space $L^1(0,T;(\frac{1}{\Phi}W^{1,\delta}(\O))')$. Write $\ol{A}$ 
for the (compact) closure of $A$ in this space.

By compactness of $\ol{A}$, the limit in $\distt$ of any sequence in $A$ also belongs to $\ol{A}$.
As $k\to\infty$ and $n\to\infty$ (in that order), we know that $\Tconc{\eps}{n}{k}\to\conc{\eps}$ 
almost-everywhere on $\O\times (0,T)$.
Hence $\Phi/\mu(\Tconc{\eps}{n}{k})\to \Phi/\mu(\conc{\eps})$ almost-everywhere on 
$\O\times (0,T)$, and thus in $\distt$ since these
functions are uniformly bounded in $L^\infty(\O\times (0,T))$. As a consequence, 
$(\Phi/\mu(\conc{\eps}))_{\eps>0}$ is
a sequence in $\ol{A}$ and thus, up to a subsequence, converges strongly in 
$L^1(0,T;(\frac{1}{\Phi}W^{1,\delta}(\O))')$.
By \eqref{conv.cN.meas}, the limit of this sequence must be $\Phi/\underline{\mu}$.
Extracting another subsequence, we can therefore assert that, as $\eps\to0$,
for almost every $t\in (0,T)$,
\[
\frac{\Phi}{\mu(\conc{\eps}(\cdot,t))}\to\frac{\Phi}{\underline{\mu}(\cdot,t)}\mbox{ strongly
in $\left(\frac{1}{\Phi}W^{1,\delta}(\O)\right)'$}.
\]
The definition of $\frac{1}{\Phi}W^{1,\delta}(\O)$ shows that, for all $Z\in L^\infty(\O)$,
$\norm{\Phi Z}_{(\frac{1}{\Phi}W^{1,\delta}(\O))'}
=\norm{Z}_{(W^{1,\delta}(\O))'}$. Then, along a 
subsequence as $\eps\to0$, for almost every $t\in (0,T)$
\begin{equation}\label{ae.conv.mucN}
\frac{1}{\mu(\conc{\eps}(\cdot,t))}\to \frac{1}{\underline{\mu}(\cdot,t)}\mbox{ strongly
in $(W^{1,\delta}(\O))'$, for all $\delta\in (1, \infty)$}.
\end{equation}

\emph{Step 2: proof of \eqref{eq:str.conv.pN.Lq}.}

From here on, we work with the subsequence along which \eqref{ae.conv.mucN} holds and explicitly denote
any other extraction of a subsequence.
Let $A_1$ be the set of $t\in (0,T)$ such that \eqref{ae.conv.mucN} holds, and
$A_2$ be the set of $t\in (0,T)$ such that, for all $q<2$,
$(\prss{\eps}(\cdot,t))_{\epsilon>0}$ is bounded in $W^{1,q}_\star(\O)$ (see Remark \ref{rem:W1l}).
Take functions $(\theta_j)_{j\ge 3}$ in $C^\infty_c(\O)$
such that $\supp(\theta_j)\cap (\discK\cup\supp(\nu))=\emptyset$, $0\le \theta_j\le 1$ and
$\theta_j\to 1$ almost-everywhere on $\O$ as $j\to\infty$.
Apply Lemma \ref{pt.bound} to $(\theta_j\prss{\eps})_{\eps>0}$ and
$E=W^{2,r}(\O)$ (see \eqref{eq:measgrisest1}), and let $A_j$ be the set
of $t\in (0,T)$ that satisfy the conclusion of the lemma.
The complement of $A=\cap_{j\in\NN}A_j$ has a zero measure.

Fix $t\in A$.
Owing to \eqref{hyp:source}, as in \cite[Step 3, proof of Theorem 2.1]{fg00} we see 
that $(a(\cdot,t)-b(\cdot,t))\nu\in (W^{1,\ell}(\O))'$ for all $\ell>2$. 
Hence by \cite[Proposition 3.2]{fg00}, there is a unique solution to
$-\dive(\frac{\K}{\underline{\mu}}\nabla P(t))=(a(\cdot,t)-b(\cdot,t))\nu$ with zero average and
homogeneous Neumann conditions in the sense
\begin{equation}\label{sol:limit.ell}
\begin{aligned}
&P(\cdot,t)\in \bigcap_{q<2}W^{1,q}_\star(\O)\mbox{ and }\forall \psi\in C^\infty(\ol{\O}),\\
&\int_\O \frac{\K(x)}{\underline{\mu}(x,t)}\nabla P(x,t)\cdot\nabla\psi(x)\ud x
=\int_\O (a-b)(x,t)\psi(x)\ud\nu(x).
\end{aligned}
\end{equation}
Note that the formulation in Fabrie--Gallou\"et \cite{fg00} is written for 
$\psi\in \bigcup_{z>d} W^{1,z}(\O)$ which, by density, is equivalent to the 
formulation above.

We first prove that, up to a subsequence (depending on $t$), $\prss{\eps}(t)\to P(t)$ 
strongly in $L^q(\O)$ for all $q<2$.
By choice of $t\in A$, there exists
a subsequence $(\prss{\eps'}(t))_{\eps'>0}$ that converges weakly 
$W^{1,q}_\star(\O)$ for all $q<2$ --- and strongly in the corresponding $L^q(\O)$ spaces --- 
toward some function $\mathcal P$. Recalling the conclusion of Lemma \ref{pt.bound}, we can
also assume that this subsequence satisfies
\[
(\theta_j\prss{\eps'}(\cdot,t))_{\eps'>0}
\mbox{ is bounded in $W^{2,r}(\O)$ for all $r<\frac{2d}{2d-1}$ and all $j\in\NN$},
\]
which shows that, for every $j\in\NN$, 
$\theta_j \prss{\eps'}(\cdot,t)\weakto \theta_j\mathcal P$ in
$W^{2,r}(\O)$ for all $r<\frac{2d}{2d-1}$. Substitute $\psi\in C^\infty(\ol{\O})$ into \eqref{eq:jdktell.pt}.
Defining $\tK_j$ by \eqref{def:tK}, with $\rho=\rho_j$ associated with $\theta_j$,
this gives (dropping the explicit mention of the $x$ variable)
\begin{align}
\int_\O (a-b)&(t)\psi\ud \nu=
\int_\O \frac{\K}{\mu(\conc{\eps'}(t))}\nabla\prss{\eps'}(t)\cdot\nabla\psi\ud x\nonumber\\
={}&\int_\O \frac{1}{\mu(\conc{\eps'}(t))}\theta_j \tK_j\nabla\prss{\eps'}(t)\cdot\nabla\psi\ud x
+\int_\O (1-\theta_j)\frac{\K}{\mu(\conc{\eps'}(t))}\nabla\prss{\eps'}(t)\cdot\nabla\psi\ud x\nonumber\\
={}&I_{\eps',j,1}+I_{\eps',j,2}.
\label{pt.conv.pN.key}
\end{align}
The tensor $\tK_j$ is Lipschitz continuous and, as $\eps'\to0$,
$\theta_j\nabla\prss{\eps'}(t)\weakto \theta_j \nabla\mathcal P$ in $W^{1,r}(\O)$ for all $r<\frac{2d}{2d-1}$.
Hence, the convergence \eqref{ae.conv.mucN} (which holds since $t\in A$) gives
\begin{equation}\label{conv.I1}
\lim_{\eps'\to0} I_{\eps',j,1}=\int_\O \frac{1}{\underline{\mu}(t)}\theta_j \tK_j\nabla\mathcal P\cdot\nabla\psi\ud x
=\int_\O \frac{\K}{\underline{\mu}(t)}\nabla \mathcal P\cdot\nabla\psi \ud x-J_j,
\end{equation}
where 
\[
J_j=\int_\O \frac{\K}{\underline{\mu}(t)}(1-\theta_j)\nabla \mathcal P\cdot\nabla\psi.
\]
Fix $q_0\in \left(1,\frac{d}{d-1}\right)$. Since $(\prss{\eps'}(t))_{\eps'>0}$ is 
bounded in $W^{1,q_0}(\O)$ and $\mathcal P\in W^{1,q_0}(\O)$, we find
$\ctel{not.dep.k}$ not depending on $j$ or $\eps$ such that
\[
|I_{\eps',j,2}|+|J_j|\le \cter{not.dep.k}\norm{1-\theta_j}_{L^{q_0'}(\O)}.
\]
Plugged into \eqref{pt.conv.pN.key}, this gives
\begin{multline}\label{end.ptconv}
\left|\int_\O (a(t)-b(t))\psi\ud \nu-\int_\O \frac{\K}{\underline{\mu}(t)}
\nabla \mathcal P\cdot\nabla\psi \ud x\right|\\
\le \left|I_{\eps',j,1}-\left(\int_\O \frac{\K}{\underline{\mu}(t)}
\nabla \mathcal P\cdot\nabla\psi \ud x-J_j\right)\right|
+ \cter{not.dep.k}\norm{1-\theta_j}_{L^{q_0'}(\O)}.
\end{multline}
Since $q_0'<\infty$, the properties of $\theta_j$ show that $\norm{1-\theta_j}_{L^{q_0'}(\O)}\to 0$ as $j\to\infty$.
Then thanks to \eqref{conv.I1}, taking the superior limit as $\eps'\to0$ and then the limit as $j\to\infty$
of \eqref{end.ptconv} shows that $\mathcal P$ satisfies \eqref{sol:limit.ell}.

We infer that $\mathcal P=P(t)$ and thus that the limit of $(\prss{\eps'}(t))_{\eps'>0}$ 
does not depend on the chosen subsequence.
In other words, the whole sequence $(\prss{\eps}(t))_{\eps>0}$ converges in $L^q(\O)$ to $P(t)$, for 
almost every $t\in (0,T)$.
By the bound in $L^\infty(0,T;L^q(\O))$ on $(\prss{\eps})_{\eps>0}$ given by \eqref{eq:measdarcyest},
the dominated convergence theorem shows that $\prss{\eps}\to P$ strongly in $\Leb{a}{q}$ for all $a<\infty$.
The convergence \eqref{eq:measpntop} imposes $P=p$ and the proof of \eqref{eq:str.conv.pN.Lq} is complete.

\emph{Step 3: proof of \eqref{eq:str.conv.pN.L2}.}

This follows from the previous convergence by a simple interpolation technique. Let $\tau\in (0,1)$ be such that
$\frac{1}{2}=\frac{\tau}{1}+\frac{1-\tau}{2^*}$, where $2^*>2$ is a Sobolev exponent (that is, such that $H^1(\O)
\hookrightarrow L^{2^*}(\O)$).
Fix $a<\infty$ and take $A\in (a,\infty)$ such that
$\frac{1}{a}=\frac{\tau}{A}+\frac{1-\tau}{\infty}$. Then
\begin{align*}
\norm{\theta\prss{\eps}-\theta p}_{\Leb{a}{2}}\le{}&
\norm{\theta\prss{\eps}-\theta p}_{\Leb{A}{1}}^\tau
\norm{\theta\prss{\eps}-\theta p}_{\Leb{\infty}{2^*}}^{1-\tau}\\
\le{}&
\norm{\theta}_{\infty}\norm{\prss{\eps}-p}_{\Leb{A}{1}}^\tau
\norm{\theta\prss{\eps}-\theta p}_{\Leb{\infty}{2^*}}^{1-\tau}.
\end{align*}
The second term in the right-hand side converges to $0$ by \eqref{eq:str.conv.pN.Lq}, and the third
term is bounded by \eqref{eq:measgrisest0}, which, combined with \eqref{eq:measpntop}, proves
in particular that $\theta p\in L^\infty(0,T;H^1(\O))$.
\end{proof}

The next lemma highlights an almost-everywhere convergence property of $(\conc{\eps})_{\eps>0}$
that is critical for obtaining strong convergence of $(\nabla\prss{\eps})_{\eps>0}$.

\begin{lemma}\label{lem:ae.cN.meas}
Assume \eqref{assum:general} and for $\eps>0$, let $(\prss{\eps},\darcy{\eps},\conc{\eps})$ be 
a solution-by-approxi\-ma\-tion to \eqref{def:sol.L2N.meas}.
Assume that \eqref{eq:meascntoc}--\eqref{eq:measuntou} hold along
a subsequence. Then, up to another subsequence,
\begin{equation}\label{eq:aeconv.c.meas}
\conc{\eps}\to c\quad\mbox{a.e.\ on $\{(x,t)\in\O\times(0,T)\,:\,|\U(x,t)|\not=0\}$}.
\end{equation}
\end{lemma}
 
\begin{proof}
By strict convexity of $\mu$ on $[0,1]$, $m:=2\min_{[0,1]}\mu''$ is strictly positive.
Lagrange's Remainder Theorem therefore gives
$\mu(\conc{\eps})-\mu(c)-(\conc{\eps}-c)\mu'(c) \geq m(\conc{\eps}-c)^{2}$.
Multiplying by $|\U|\geq 0$ and integrating yields
\begin{multline*}
\int_0^T\int_\O\left(\mu(\conc{\eps}(x,t))|\U(x,t)|- \mu(c(x,t))|\U(x,t)|\right)\ud x\ud t\\
- \int_0^T\int_\O(\conc{\eps}- c)(x,t)\mu'(c(x,t))|\U(x,t)|\ud x\ud t \\
\quad \ge m\int_0^T\int_\O (\conc{\eps}(x,t)-c(x,t))^2|\U(x,t)|\ud x\ud t.
\end{multline*}
By \eqref{eq:meascntoc} and \eqref{conv.cN.meas}, since $|\U|\in L^1(\O\times(0,T))$,
we pass to the limit in the left-hand side to obtain
\begin{multline*}
\int_0^T\int_\O \left(\ol{\mu}(x,t)|\U(x,t)|- \mu(c(x,t))|\U(x,t)|\right)\ud x\ud t\\
\ge m \limsup_{\eps\to0}\int_0^T\int_\O (\conc{\eps}(x,t)-c(x,t))^2|\U(x,t)|\ud x\ud t.
\end{multline*}
Thanks to \eqref{eq:measdarcylimit} and \eqref{eq:uid1.meas} we have $\ol{\mu}\U=\mu(c)\U$. Taking
the norms, we see that left-hand side vanishes. This shows that
$(\conc{\eps}-c)^2|\U|\to 0$ in $L^1(\O\times(0,T))$, and therefore almost-everywhere on $\O\times (0,T)$
up to a subsequence.
\end{proof}

\begin{remark}
The main purpose of this almost-everywhere convergence of $(\conc{\eps})_{\eps>0}$ is
to prove the convergence of $(\darcy{\eps})_{\eps>0}$.

Lemma \ref{lem:ae.cN.meas} is no longer valid if $\mu$ is constant.
However, in that case, the system is decoupled: the pressure does not depend on
the concentration (and then does not even depend on $\eps$),
and there are no difficulties with the convergence
of $\darcy{\eps}$ as it does not depend on $\eps$.
\end{remark}

\begin{lemma}\label{lem:gradstrong.meas}
Assume \eqref{assum:general}. For $\eps>0$, let $(\prss{\eps},\darcy{\eps},\conc{\eps})$ be 
a solution-by-approxi\-ma\-tion to \eqref{def:sol.L2N.meas}.
Assume that \eqref{eq:meascntoc}--\eqref{eq:measuntou} hold along
a subsequence. Then along the same subsequence,
\begin{equation}\label{eq:gradpstrong.meas}
\nabla\prss{\eps}\to\nabla p\quad\mbox{ strongly in $\Leb[d]{a}{q}$ for all $a<\infty$ and all $q<\frac{d}{d-1}$.}
\end{equation}
\end{lemma}

\begin{proof}
\emph{Step 1: strong convergence of localised functions.}

Let $\rho\in C^\infty_c(\O)$ such that $\supp(\rho)\cap(\discK\cup\supp(\nu))=\emptyset$ and $\rho\ge 0$.
We want to prove that as $\eps\to0$,
\begin{equation}\label{eq:gradpstrong.rho}
\sqrt{\rho}\nabla\prss{\eps}\to\sqrt{\rho} \nabla p\mbox{ strongly in $\Leb[d]{2}{2}$.}
\end{equation}
Let $\psi\in L^1(0,T;W^{1,q}(\O))$ for some $q>d$, and take $\rho\psi$ as a test
function in the equation \eqref{eq:jdktell} satisfied by $\darcy{\eps}$. Since $\supp(\rho)\cap\supp(\nu)=0$, the
source terms disappear and we find that
\begin{equation}\label{eq:uNu.1}
\int_0^T\int_\O\darcy{\eps}(x,t)\cdot\nabla(\rho\psi)(x,t)\ud x\ud t=0.
\end{equation}
Let $U$ be an open set in $\O$ such that $\supp(\rho)\subset U$ and $\ol{U}\cap
(\discK\cup\supp(\nu))=\emptyset$. Let $\theta\in C^\infty_c(\O)$ be such that $\theta=1$ on
$\ol{U}$ and $\supp(\theta)\cap (\discK\cup\supp(\nu))=\emptyset$.
Applying \eqref{eq:measgrisest0}, we see that $\darcy{\eps}\in L^2(0,T;L^2(U))$
and $\prss{\eps}-p\in L^2(0,T;H^1(U))$. Taking a sequence 
$(\psi_j)_{j\in\NN}\subset L^1(0,T; W^{1,q}(\O))$ for some $q>d$ and such that
$\psi_j\to\prss{\eps}-p$ in $L^2(0,T; H^1(U))$, we pass to the limit in \eqref{eq:uNu.1}
to see that this relation still holds with $\prss{\eps}-p$ instead of $\psi$. Expanding, we obtain
\begin{multline*}
-\int_0^T\int_\O \rho(x)\darcy{\eps}(x,t)\cdot\nabla(\prss{\eps}-p)(x,t)\ud x\ud t\\
=
\int_0^T\int_\O (\prss{\eps}(x,t)-p(x,t))\darcy{\eps}(x,t)\cdot\nabla\rho(x)\ud x\ud t.
\end{multline*}
By the choice of $\theta$ above, this can be written as
\begin{align} \label{eq:uNu.2}
-\int_0^T\int_\O \rho(x) &\darcy{\eps}(x,t)\cdot \nabla(\prss{\eps}-p)(x,t)\ud x\ud t\nonumber\\
={}&\int_0^T\int_\O (\theta(x)\prss{\eps}(x,t)-\theta(x)p(x,t))
\theta(x)\darcy{\eps}(x,t)\cdot \nabla \rho(x)\ud x\ud t\nonumber\\
\leq{}& \cter{cst:uN} \norm{\theta\prss{\eps}-\theta p}_{\Leb{2}{2}},
\end{align}
where the existence of $\ctel{cst:uN}$ (not depending on $\eps$) is ensured by 
\eqref{eq:measgrisest0}, which shows that $(\theta\darcy{\eps})_{\eps>0}$ is bounded in $\Leb[d]{2}{2}$.
Now use the definition of $\darcy{\eps}=-\frac{\K}{\mu(\conc{\eps})}\nabla\prss{\eps}$,
estimate \eqref{eq:uNu.2} and the properties of $\theta$ to write
\begin{align}
\frac{k_\ast}{\mu^\ast}&\norm{\sqrt{\rho}\nabla(\prss{\eps}-p)}_{\Leb[d]{2}{2}}^2\nonumber\\
\le{}&\int_0^T\int_\O \rho(x)\frac{\K(x)}{\mu(\conc{\eps}(x,t))}\nabla(\prss{\eps}-p)(x,t)
\cdot\nabla(\prss{\eps}-p)(x,t)\ud x\ud t\nonumber\\
={}&-\int_0^T\int_\O \rho(x)\darcy{\eps}(x,t)\cdot\nabla(\prss{\eps}-p)(x,t)\ud x\ud t\nonumber\\
&-\int_0^T\int_\O \rho(x)\frac{\K(x)}{\mu(\conc{\eps}(x,t))}\nabla p(x,t)
\cdot\nabla(\prss{\eps}-p)(x,t)\ud x\ud t\nonumber\\
\le{}&\cter{cst:uN} \norm{\theta\prss{\eps}-\theta p}_{\Leb{2}{2}}\nonumber\\
&-\int_0^T\int_\O \rho(x)\frac{\K(x)}{\mu(\conc{\eps}(x,t))}\theta(x)\nabla p(x,t)
\cdot\theta(x)\nabla(\prss{\eps}-p)(x,t)\ud x\ud t.
\label{cv.DpN.2}
\end{align}
By \eqref{eq:aeconv.c.meas}, $\mu(\conc{\eps})\to \mu(c)$ almost-everywhere on 
$\{|\U|\not=0\}=\{|\nabla p|\not=0\}$.
Hence, by the dominated convergence theorem and \eqref{eq:measgrisest0},
$(\frac{\K}{\mu(\conc{\eps})}\theta\nabla p)_{\eps>0}$ converges strongly in $\Leb[d]{2}{2}$.
Using \eqref{eq:measgrisest0} and \eqref{eq:measpntop} we also have
$\theta\nabla\prss{\eps}\to\theta\nabla p$ weakly in $\Leb[d]{2}{2}$.
Hence, the last term in \eqref{cv.DpN.2} tends to $0$ as $\eps\to0$.
Taking the superior limit of this estimate and using \eqref{eq:str.conv.pN.L2} shows that
\eqref{eq:gradpstrong.rho} holds.

\emph{Step 2: conclusion.}

Since \eqref{eq:gradpstrong.rho} is satisfied for all nonnegative $\rho\in C^\infty_c(\O)$ whose support
does not intersect the closed set $\discK\cup\supp(\nu)$, and since this set has a zero Lebesgue measure,
up to a subsequence we can assume that $\nabla\prss{\eps}\to\nabla p$ almost-everywhere on $\O\times (0,T)$.
The convergence \eqref{eq:gradpstrong.meas} then follows from the Vitali theorem and the
bound \eqref{eq:measdarcyest} on $(\nabla\prss{\eps})_{\eps>0}$ in $\Leb[d]{\infty}{q}$ for all $q<d/(d-1)$.
\end{proof}

The strong convergence of the Darcy velocity and of $(\conc{\eps}\darcy{\eps})_{\eps>0}$ is then straightforward.
\begin{lemma}\label{lem:strongdarcy.meas}
Assume \eqref{assum:general}. For $\eps>0$, let $(\prss{\eps},\darcy{\eps},\conc{\eps})$ be a 
solution-by-approxi\-ma\-tion of \eqref{def:sol.L2N.meas}.
Assume that \eqref{eq:meascntoc}--\eqref{eq:measuntou} hold along
a subsequence. Then along the same subsequence,
\begin{equation}\label{eq:strongdarcy.meas}
\darcy{\eps}\to\U \quad\mbox{strongly in $\Leb[d]{a}{q}$ for all $a<\infty$ and all $q<\frac{d}{d-1}$.}
\end{equation}
\end{lemma}
\begin{proof}
The almost-everywhere convergence \eqref{eq:aeconv.c.meas} of $\conc{\eps}$ gives
$\mu(\conc{\eps})\U\to\mu(c)\U$ almost-everywhere on $\O\times(0,T)$. Since $\U\in\Leb[d]{a}{q}$
for all $a<\infty$ and $q<d/(d-1)$, this convergence also holds in $\Leb[d]{a}{q}$ by
dominated convergence. Thanks to \eqref{eq:measdarcylimit} and \eqref{eq:gradpstrong.meas}, 
letting $\eps\to0$ gives
\begin{align*}
\mu_{\ast}\norm{\darcy{\eps}-\U}_{\Leb[d]{a}{q}}
\leq{}&\norm{\mu(\conc{\eps})\darcy{\eps}-\mu(\conc{\eps})\U}_{\Leb[d]{a}{q}} \\
		={}& \norm{-\K\nabla\prss{\eps}-\mu(\conc{\eps})\U}_{\Leb[d]{a}{q}}\\
		&\to \norm{-\K\nabla p - \mu(c)\U}_{\Leb[d]{a}{q}}=0. \qedhere
\end{align*}
\end{proof}
\begin{corollary} \label{cor:custrong}
Assume \eqref{assum:general}. For $\eps>0$, let $(\prss{\eps},\darcy{\eps},\conc{\eps})$ be a 
solution-by-approxi\-ma\-tion of \eqref{def:sol.L2N.meas}.
Assume that \eqref{eq:meascntoc}--\eqref{eq:measuntou} hold along
a subsequence. Then along the same subsequence,
\begin{equation*}
\conc{\eps}\darcy{\eps}\to c\U \quad\mbox{strongly in $\Leb[d]{a}{q}$ for all $a<\infty$ and all $q<\frac{d}{d-1}$.}
\end{equation*}
\end{corollary}
\begin{proof}
Write $\conc{\eps}\darcy{\eps}-c\U = \conc{\eps}(\darcy{\eps}-\U) + (\conc{\eps}-c)\U$. 
Owing to \eqref{eq:strongdarcy.meas} and $0\leq\conc{\eps}\leq1$, the first term tends to 
$0$ in $\Leb[d]{a}{q}$ as $\eps\to0$. For the second term, use \eqref{eq:aeconv.c.meas} 
and the fact that $\U\in \Leb[d]{a}{q}$. 
\end{proof}

\subsection{Passing to the limit in the concentration equation} \label{ssec:passlimitconc}

The proof that $(p,\U,c)$ satisfies \eqref{eq:measdgradc} and \eqref{eq:measpara} follows
from the next two lemmas, which address the regularity and convergence of the 
diffusion-dispersion term.

\begin{lemma}\label{lem:c.has.grad.meas}
Assume \eqref{assum:general} and for $\eps>0$, let $(\prss{\eps},\darcy{\eps},\conc{\eps})$ be a 
solution-by-approxi\-ma\-tion to \eqref{def:sol.L2N.meas}.
Assume that \eqref{eq:meascntoc}--\eqref{eq:measuntou} hold along a subsequence.
Then the function $c$ defined by \eqref{eq:meascntoc} has a $\{ |\U|>0\}$-gradient and
\begin{equation} \label{eq:regdlimit.meas}
\D[\circ]{\cdot}{\U}\gradc{|\U|}{0} \in \Leb[d]{2}{r}\mbox{ for all $r<\frac{2d}{2d-1}$}.
\end{equation}
\end{lemma}
\begin{proof}
From \eqref{eq:strongdarcy.meas} and the partial converse to the dominated convergence theorem, 
up to a subsequence $\darcy{\eps}\to\U$ almost everywhere on $\O\times(0,T)$. 
Let $(\eta_{i})_{i\in\NN}$ be a sequence in $\RR$ with $\eta_i \to 0^{+}$ as $i \to \infty$ 
and such that for every $i\in\NN$, $\meas(\{|\U|=\eta_i \})=0$ 
(existence of such a sequence is guaranteed by Lemma \ref{lem:levelset}). 
On the set $\{ |\darcy{\eps}|>\eta_i \}$ we have 
\begin{equation*}
\D[\eps]{x}{\darcy{\eps}}\nabla\conc{\eps}\cdot\nabla\conc{\eps}
\geq\mathrm{min}(\dl,\dt)\phi_{\ast}\eta_i|\nabla\conc{\eps}|^{2}.
\end{equation*}
Since $(\D[\eps]{\cdot}{\darcy{\eps}}\nabla\conc{\eps}\cdot\nabla\conc{\eps})_{\eps>0}$ 
is bounded in $\Leb{1}{1}$ (see \eqref{eq:measenergyest}), it follows that
\begin{equation*}
(\mathbf{1}_{\{ |\darcy{\eps}|>\eta_i \}}\nabla\conc{\eps})_{\eps>0} \quad 
\mbox{is bounded in $\Leb[d]{2}{2}$ for all $i\in\NN$}.
\end{equation*}
After performing a diagonal extraction upon the index $i$, we infer the existence
of $\chi_{\eta_i} \in \Leb[d]{2}{2}$ such that, up to a subsequence 
not depending on $i$, 
\begin{equation}\label{eq:gradcweak.meas}
\ch_{\{|\darcy{\eps}|>\eta_i\}}\nabla\conc{\eps}\weakto \chi_{\eta_i}\quad\mbox{ weakly in
$\Leb[d]{2}{2}$.}
\end{equation}
The hypotheses of Definition \ref{def:ugrad} are therefore satisfied and so 
$c$ has a $\{ |\U|>0 \}$-gradient.

To prove \eqref{eq:regdlimit.meas}, we begin by using the same splitting trick
as in our previous work \cite[Section 4.3]{dt14} by writing
\begin{equation}\label{decom.DNcN}
\ch_{\{|\darcy{\eps}|>\eta_i\}}\D[\eps]{\cdot}{\darcy{\eps}}\nabla\conc{\eps}=\Dsq[\eps]{\cdot}{\darcy{\eps}}
\left(\ch_{\{|\darcy{\eps}|>\eta_i\}}\Dsq[\eps]{\cdot}{\darcy{\eps}}\nabla\conc{\eps}\right)
\end{equation}
and applying Lemma \ref{lem:ws} once to each term in the right-hand side product.
By \eqref{eq:strongdarcy.meas} with $a=q=1$ and the estimate \eqref{eq:Dsq:bound} on $\DD_\eps^{1/2}$,
Lemma \ref{lem:h} yields
\begin{equation}\label{cv.Dsq.1}
\Dsq[\eps]{\cdot}{\darcy{\eps}}\to \Dsq[\circ]{\cdot}{\U}\quad\mbox{ strongly in $\Leb[d\times d]{2}{2}$}.
\end{equation}
Since $(\Dsq[\eps]{\cdot}{\darcy{\eps}}\nabla\conc{\eps})_{\eps>0}$ is bounded in $\Leb[d]{2}{2}$
(see \eqref{eq:measenergyest}), the weak convergence \eqref{eq:gradcweak.meas} enables
us to apply Lemma \ref{lem:ws} with $r_1=r_2=s_2=s_2=2$ and $a=b=2$,
to $w_\eps=$ components of $\Dsq[\eps]{\cdot}{\darcy{\eps}}$ and $v_\eps=$ components of 
$\ch_{\{|\darcy{\eps}|>\eta_i\}}\nabla\conc{\eps}$. This gives
\[
\ch_{\{|\darcy{\eps}|>\eta_i\}}\Dsq[\eps]{\cdot}{\darcy{\eps}}\nabla\conc{\eps}\weakto 
\Dsq[\circ]{\cdot}{\U}\chi_{\eta_i}=\Dsq[\circ]{\cdot}{\U}\gradc{|\U|}{\eta_i}
\]
weakly in $\Leb[d]{2}{2}$.
This weak convergence and \eqref{cv.Dsq.1} enable us to re-use Lemma \ref{lem:ws}
with $w_\eps=$ components of $\Dsq[\eps]{\cdot}{\darcy{\eps}}$ and $v_\eps=$
components of $\ch_{\{|\darcy{\eps}|>\eta_i\}}\Dsq[\eps]{\cdot}{\darcy{\eps}}\nabla\conc{\eps}$. Owing to the decomposition
\eqref{decom.DNcN}, the bound \eqref{eq:measdest} then shows that, for any $r<\frac{2d}{2d-1}$,
\begin{equation}\label{cv.DcN.step1}
\ch_{\{|\darcy{\eps}|>\eta_i\}}\D[\eps]{\cdot}{\darcy{\eps}}\nabla\conc{\eps}\weakto\D[\circ]{\cdot}{\U}\gradc{|\U|}{\eta_i}
\mbox{ weakly in $\Leb[d]{2}{r}$}.
\end{equation}
In particular, this shows that 
\begin{multline*}
\norm{\D[\circ]{\cdot}{\U}\gradc{|\U|}{\eta_{i}}}_{\Leb[d]{2}{r}}\\
= \norm{\ch_{\{|\U|>\eta_{i}\}}\D[\circ]{\cdot}{\U}\gradc{|\U|}{0}}_{\Leb[d]{2}{r}}
\leq \cter{measd.r}.
\end{multline*}
Now $\ch_{\{ |\U| > \eta_i \}} \to \ch_{\{ |\U| > 0 \}}$ almost-everywhere as $i \to \infty$, 
so by the Fatou lemma (applied twice),
\begin{align*}
\int_{0}^{T} &\norm{\D[\circ]{\cdot}{\U(\cdot,t)}\gradc{|\U|}{0}(\cdot,t)}_{L^{r}(\O)^{d}}^{2}\ud t \\
&\leq \int_{0}^{T} \liminf_{i \to \infty} \norm{\ch_{\{ |\U| > \eta_i \}}
\D[\circ]{\cdot}{\U(\cdot,t)}\gradc{|\U|}{0}(\cdot,t)}_{L^{r}(\O)^{d}}^{2}\ud t\\
&\leq \liminf_{i \to \infty} \int_{0}^{T} \norm{\ch_{\{ |\U| > \eta_i \}}
\D[\circ]{\cdot}{\U(\cdot,t)}\gradc{|\U|}{0}(\cdot,t)}_{L^{r}(\O)^{d}}^{2}\ud t\leq \cter{measd.r}.
\qedhere
\end{align*}
\end{proof}

\begin{lemma} \label{lem:ddtlimit.meas}
Assume \eqref{assum:general} and for $\eps>0$ let $(\prss{\eps},\darcy{\eps},\conc{\eps})$ be a 
solution-by-approxi\-ma\-tion to \eqref{def:sol.L2N.meas}.
Assume that \eqref{eq:meascntoc}--\eqref{eq:measuntou} hold along a subsequence.
Then along the same subsequence,
\begin{equation} \label{eq:ddtconv_meas}
	\begin{gathered} 
	\D[\eps]{\cdot}{\darcy{\eps}}\nabla\conc{\eps}\weakto\D[\circ]{\cdot}{\U}\gradc{|\U|}{0} 
	\quad \mbox{weakly in $\Leb[d]{2}{r}$}\\ 
	\mbox{for all $r<\frac{2d}{2d-1}$.} 
	\end{gathered}
\end{equation}	
\end{lemma}
\begin{proof}
Let $\bpsi \in \Leb[d]{2}{r'}$ and $i\in\NN$. Write
\begin{align}
\int_{0}^{T}\int_{\O}&\D[\eps]{x}{\darcy{\eps}(x,t)}\nabla\conc{\eps}(x,t)\cdot\bpsi(x,t) \ud x \ud t \nonumber\\
={}& \int_{0}^{T}\int_{\O}\ch_{\{ |\darcy{\eps}| > \eta_{i}\}}
\D[\eps]{x}{\darcy{\eps}(x,t)}\nabla\conc{\eps}(x,t)\cdot\bpsi(x,t) \ud x \ud t \nonumber\\
&+ \int_{0}^{T}\int_{\O}\ch_{\{ |\darcy{\eps}| \leq \eta_{i}\}}
\D[\eps]{x}{\darcy{\eps}(x,t)}\nabla\conc{\eps}(x,t) \cdot \bpsi(x,t) \ud x \ud t \nonumber\\
={}& T_{1}(i,\eps) + T_{2}(i,\eps),
 \label{last.cv.decomp.0}
\end{align}
and
\begin{align}
\int_{0}^{T}\int_{\O}&\D[\circ]{x}{\U(x,t)}\gradc{|\U|}{0}(x,t)\cdot\bpsi(x,t) \ud x \ud t \nonumber\\
={}& \int_{0}^{T}\int_{\O}\ch_{\{|\U|>\eta_{i}\}}\D[\circ]{x}{\U(x,t)}\gradc{|\U|}{0}(x,t)\cdot\bpsi(x,t)\ud x\ud t \nonumber\\
&+ \int_{0}^{T}\int_{\O}\ch_{\{0<|\U|\leq\eta_{i}\}}\D[\circ]{x}{\U(x,t)}\gradc{|\U|}{0}(x,t)\cdot\bpsi(x,t)\ud x \ud t\nonumber\\
={}& L_{1}(i) + L_{2}(i).
 \label{last.cv.decomp.1}
\end{align}
Using \eqref{cv.DcN.step1} we obtain $\lim_{\eps\to0} T_1(i,\eps)=L_1(i)$.
For $T_2(i,\eps)$, use the estimate \eqref{eq:measenergyest} on $\Dsq[\eps]{\cdot}{\darcy{\eps}}\nabla\conc{\eps}$
and the estimate \eqref{eq:Dsq:bound} on $\mathbf{D}^{1/2}_\eps$ to obtain
\begin{align*}
|T_{2}&(i,\eps)| \\
\leq{}& \int_{0}^{T}\int_{\O}\ch_{\{ |\darcy{\eps}|\leq\eta_{i}\}}|
\Dsq[\eps]{x}{\darcy{\eps}(x,t)}\nabla\conc{\eps}(x,t)\cdot\Dsq[\eps]{x}{\darcy{\eps}(x,t)}\bpsi(x,t)|\ud x \ud t \\
\leq{}& \norm{\Dsq[\eps]{\cdot}{\darcy{\eps}}\nabla\conc{\eps}}_{\Leb[d]{2}{2}}
\norm{\ch_{\{ |\darcy{\eps}|\leq \eta_i \}}\Dsq[\eps]{\cdot}{\darcy{\eps}}\bpsi}_{\Leb[d]{2}{2}} \\
\leq{}& \cter{dmeas} \phi_\ast^{-1/2}\left(\eps+\max(\dl,\dt)\eta_i\right)^{1/2}\norm{\bpsi}_{\Leb[d]{2}{2}}.
\end{align*}
This shows that
\begin{equation*}
\lim_{i\to\infty}\limsup_{\eps\to0} T_2(i,\eps) = 0.
\end{equation*}
For $L_2(i)$, use \eqref{eq:regdlimit.meas} and write
\begin{align*}
|L_{2}(i)| 
&\leq \int_{0}^{T}\int_{\O}\ch_{\{ 0 < |\U| \leq \eta_i \}}|\D[\circ]{x}{\U}\gradc{|\U|}{0} \cdot \bpsi | \ud x \ud t \\
&\leq \norm{\D[\circ]{\cdot}{\U}\gradc{|\U|}{0}}_{\Leb[d]{2}{r}}
\norm{\ch_{\{ 0 < |\U| \leq \eta_i \}}\bpsi}_{\Leb[d]{2}{r'}} \\
&\leq \cter{measd.r}\norm{\ch_{\{ 0 < |\U| \leq \eta_i \}}\bpsi}_{\Leb[d]{2}{r'}}
\to 0 \quad \mbox{ as }i \to \infty.
\end{align*}
Then
\begin{align*}
\bigg|\int_{0}^{T}\int_{\O}\D[\eps]{x}{\darcy{\eps}(x,t)}&\nabla\conc{\eps}(x,t)\cdot\bpsi(x,t)\ud x \ud t \\
&- \int_{0}^{T}\int_{\O}\D[\circ]{x}{\U(x,t)}\gradc{|\U|}{0}(x,t)\cdot\bpsi(x,t)\ud x\ud t \bigg| \\
={}& \left|T_{1}(i,\eps) + T_{2}(i,\eps) - (L_{1}(i) + L_{2}(i))\right| \\
\leq{}& |T_{1}(i,\eps) - L_{1}(i)| + |T_{2}(i,\eps)| + |L_{2}(i)|.
\end{align*}
Then taking (in this order) the limit superior as $\eps\to0$ and the limit 
as $i\to\infty$, we conclude that as $\eps\to0$
\begin{multline*}
\int_{0}^{T}\int_{\O}\D[\eps]{x}{\darcy{\eps}(x,t)}\nabla\conc{\eps}(x,t)\cdot\bpsi(x,t)\ud x\ud t  \\
\to \int_{0}^{T}\int_{\O}\D[\circ]{x}{\U(x,t)}\gradc{|\U|}{0}(x,t)\cdot\bpsi(x,t)\ud x\ud t.\qedhere
\end{multline*}
\end{proof}

The proof of Theorem \ref{th:main} is now easy to complete. 
Equation \eqref{eq:lim.dt}, Corollary \ref{cor:custrong} and Lemma \ref{lem:ddtlimit.meas}
enable us to take the limit of \eqref{eq:jdktpara}, thus proving \eqref{eq:measpara}.
To prove the last two parts of \eqref{eq:measc}, that is,
$c\in L^{\infty}(0,T; L^{1}(\O,\nu))$ and $0\leq c(x,t)\leq 1$ for $\nu$-almost-every
$x\in\O$ and for almost-every $t\in(0,T)$, follow exactly the same argument as that employed
by Fabrie--Gallou\"et \cite[Lemma 5.1]{fg00}.

\begin{remark}
Note that we can use exactly the same method as in Lemmas \ref{lem:c.has.grad.meas} 
and \ref{lem:ddtlimit.meas} to show that
\begin{gather*}
\Dsq[\eps]{\cdot}{\darcy{\eps}}\nabla\conc{\eps}\weakto\Dsq[\circ]{\cdot}{\U}\gradc{|\U|}{0} \mbox{ weakly in $\Leb[d]{2}{2}$, and}\\
\darcy{\eps}\cdot\nabla\conc{\eps}\rightharpoonup \U\cdot\gradc{|\U|}{0} \mbox{ weakly in $\Leb{2}{r}$ for
all $r<\frac{2d}{2d-1}$}.
\end{gather*}
The latter is particularly relevant in the nonconservative formulation of \eqref{eq:conc}, 
in which a term of that form appears. 
\end{remark}


\appendix
\section{Properties of the concentration gradient} \label{sec:appa}

The results in this appendix attest to the consistency of Definition \ref{def:ugrad}. 
Lemmas \ref{lem:levelset} and \ref{lem:char} give the necessary background for 
Proposition \ref{prop:welldefined}, which makes precise the dependence of the
$\{ v > 0 \}$-gradient (Definition \ref{def:ugrad}) on the sequences necessary
to define it. 

\begin{lemma} \label{lem:levelset}
Let $(\O, \Sigma, \mu)$ be a $\sigma$-finite measure space and $f: \Omega \to \RR$ 
be measurable. For almost every $k \in \RR$, 
\begin{equation}
\mu(\{ f = k	 \}) = 0. \label{eq:level}
\end{equation}
\end{lemma}
\begin{proof}
We use the Fubini--Tonelli theorem to measure the graph $G=\{(x,f(x))\,:\,x\in \O\}$ of $f$
in $\Omega\times\RR$. Denote by $\lambda$ the Lebesgue measure on $\RR$.
For a given $k\in\RR$, the slice $G_k$ of $G$ at $k$ in the first direction is
$G_k=\{x\in\O\,:\,f(x)=k\}$. For a fixed
$x\in\O$, the slice $G^x$ of $G$ at $x$ in the second direction is
$G^x=\{f(x)\}$. By Fubini--Tonelli, we therefore have
\[
\int_{\RR}\mu(\{x\in\O\,:\,f(x)=k\}) d\lambda(k)=\int_\O \lambda(\{f(x)\}) d\mu(x).
\]
Since $\lambda(\{f(x)\})=0$ for all $x\in\O$, this shows that $\int_{\RR}\mu(\{x\in\O\,:\,f(x)=k\}) d\lambda(k)=0$
and the conclusion follows. \end{proof}

\begin{lemma}\label{lem:char}
Let $(\O, \Sigma, \mu)$ be a $\sigma$-finite measure space and for every 
$\eps>0$ let $f_\eps :\O\to\RR$ be measurable. Suppose there is a measurable 
function $f:\O\to\RR$ such that $f_\eps \to f$ almost-everywhere as $\eps\to0$. 
Then for every $k\in\RR$ satisfying \eqref{eq:level}, 
\begin{equation*}
\mathbf{1}_{\{ f_\eps > k\}} \to \mathbf{1}_{\{ f > k\}} \quad \mbox { a.e.\ as }\eps\to0.
\end{equation*}
\end{lemma}
\begin{proof}
Take $k$ such that $A=\{x\in \O\,:\,f(x)=k\}$ is null, and let
$B$ be the null set $\{x\in\O\,:\,f_\eps(x)\not\to f(x)\}$. If $x\not\in A\cup B$ we have
either $f(x)>k$ or $f(x)<k$. In each respective case, for $\eps$ sufficiently small, $f_\eps(x)>k$
(respectively $f_\eps(x)<k$) and thus $\mathbf{1}_{\{ f_\eps > k\}} = \mathbf{1}_{\{ f > k\}}$. 
\end{proof}

\begin{proposition} \label{prop:welldefined}
Let $f,v \in \Leb{2}{2}$ be such that $f$ has a $\{ v > 0 \}$-gradient in the 
sense of Definition \ref{def:ugrad}. Then
\begin{enumerate}[(i)]
\item The $\{v>0\}$-gradient is independent of the choice of sequence $(\eta_{i})_{i\in\NN}$. 
\item $\gradf{v}{\eta_i }=0$ on $\{v\leq\eta_i\}$.
\item The $\{ v > 0 \}$-gradient is independent of the choice of sequence $(v_\eps)_{\eps>0}$.
\end{enumerate}
\end{proposition}

\begin{proof}
Fix the sequences $(f_\eps)_{\eps>0}, (v_\eps)_{\eps>0}$ in Definition \ref{def:ugrad} and 
let $(\eta_{i})_{i\in\NN}$ and $(\zeta_{i})_{i\in\NN}$ 
be two sequences of real numbers such that for every $i\in\NN$, 
\begin{equation*}
\meas(\{ v = \eta_i \}) = \meas(\{ v = \zeta_i \}) = 0.
\end{equation*}
Let $\chi_{\eta_i}, \chi_{\zeta_i} \in \Leb[d]{2}{2}$ be such that as $\eps\to0$,
\begin{align*}
\mathbf{1}_{\{ v_\eps > \eta_i \}}\nabla f_\eps &\rightharpoonup \chi_{\eta_i} 
\quad \mbox{weakly in $\Leb[d]{2}{2}$, and}\\
\mathbf{1}_{\{ v_\eps > \zeta_i \}}\nabla f_\eps &\rightharpoonup \chi_{\zeta_i} \quad \mbox{weakly in $\Leb[d]{2}{2}$}.
\end{align*}
It suffices to show that for any $i\in\NN$, 
$\chi_{\eta_{i}}=\chi_{\zeta_{i}}$ on $\{ v > \eta_i \}\cap\{ v > \zeta_i \}$.
Without loss of generality, assume that $\eta_i > \zeta_i$ so that 
$\{ v > \eta_i \}\cap\{ v > \zeta_i \} = \{ v > \eta_i \}$.
We have
\begin{equation}\label{eq:relation}
\ch_{\{ v_\eps > \eta_i \}}\nabla f_\eps = \ch_{\{ v_\eps > \eta_i \}}\ch_{\{ v_\eps > \zeta_i \}}\nabla f_\eps. 
\end{equation}
Thanks to Lemma \ref{lem:char}, as $\eps\to0$,
\begin{align}
\ch_{\{ v_\eps > \eta_{i} \}}\nabla f_\eps &\rightharpoonup \chi_{\eta_i}\quad\mbox{weakly in }\Leb[d]{2}{2}, \nonumber\\
\ch_{\{ v_\eps > \eta_{i} \}} &\to \ch_{\{ v > \eta_{i} \}}\quad \mbox{a.e.\ on }\O\times(0, T), \mbox{ and} \label{eq:second}\\
\ch_{\{ v_\eps > \zeta_{i} \}}\nabla f_\eps &\rightharpoonup \chi_{\zeta_{i}}\quad\mbox{weakly in }\Leb[d]{2}{2}\nonumber.
\end{align}
Passing to the weak limit in $\Leb[d]{2}{2}$ on \eqref{eq:relation} shows that on $\{ v > \eta_i \}$,
$\chi_{\eta_{i}}=\chi_{\zeta_{i}}$ in $\Leb[d]{2}{2}$,
which proves (i). 

For (ii), we have
\begin{align}
\mathbf{1}_{\{ v_\eps > \eta_i \}}\nabla f_\eps &\rightharpoonup \gradf{v}{\eta_i } \quad \mbox{ weakly in $\Leb[d]{2}{2}$, and}\label{eq:a}\\
\mathbf{1}_{\{ v_\eps > \eta_i \}} &\to \mathbf{1}_{\{ v > \eta_i \}} \quad \mbox{ in $\Leb{2}{2}$}. \nonumber
\end{align}
Then 
\begin{equation}
\mathbf{1}_{\{ v_\eps > \eta_i \}}\nabla f_\eps = \mathbf{1}_{\{ v_\eps > \eta_i \}}
\mathbf{1}_{\{ v_\eps > \eta_i \}}\nabla f_\eps \to \mathbf{1}_{\{ v > \eta_i \}}\gradf{v}{\eta_i } 
 \mbox { in }\mathcal{D'}(\O\times(0,T)). \label{eq:b}
\end{equation}
Comparing \eqref{eq:a} and \eqref{eq:b}, we see that
$\gradf{v}{\eta_i}=\ch_{\{ v > \eta_i \}}\gradf{v}{\eta_i}$ in $\mathcal{D'}(\O\times(0,T))$,
which shows that $\gradf{v}{\eta_i } = 0$ on $\{ v \leq \eta_i \}$. 

For (iii), fix the sequence $(f_\eps)_{\eps>0}$ and let $(v_\eps)_{\eps>0}$ and 
$(\bar{v}_\eps)_{\eps>0}$ be two sequences in $\Leb{2}{2}$ such that as $\eps\to0$,
$v_\eps\to v$ and $\bar{v}_\eps\to v$, both almost-everywhere $\O\times(0, T)$.
Let $\eta_i > 0$ be such that $\meas(\{ v = \eta_i \}) = 0$ and suppose that 
there are functions $\chi_{\eta_i }, \ol{\chi}_{\eta_i } \in \Leb[d]{2}{2}$ such that, as $\eps\to0$,
\begin{align*}
\ch_{\{ v_\eps > \eta_i \}}\nabla f_\eps &\rightharpoonup \chi_{\eta_i } 
\quad \mbox{weakly in $\Leb[d]{2}{2}$, and}\\
\ch_{\{ \bar{v}_{\eps} > \eta_i \}}\nabla f_\eps &\rightharpoonup \ol{\chi}_{\eta_i } 
\quad \mbox{weakly in $\Leb[d]{2}{2}$}.
\end{align*}
Observe that by Lemma \ref{lem:char},
\begin{multline*}
(\ch_{\{ v_\eps > \eta_i \}} - \ch_{\{ \bar{v}_\eps > \eta_i \}})\nabla f_\eps \\
= \ch_{\{ \bar{v}_\eps \leq \eta_i \}}(\ch_{\{ v_\eps > \eta_i\}}\nabla f_\eps) - 
\ch_{\{ v_\eps \leq \eta_i \}}(\ch_{\{ \bar{v}_\eps > \eta_i \}}\nabla f_\eps)\\
\weakto \ch_{\{ v \leq \eta_{i}\}}\chi_{\eta_{i}} - 
\ch_{\{ v \leq \eta_{i}\}}\bar{\chi}_{\eta_i } \mbox{ weakly in $\Leb[d]{2}{2}$ as }\eps\to0.
\end{multline*}
By (ii), the last term vanishes, which shows that $\chi_{\eta_i } = \ol{\chi}_{\eta_i }$ 
in $\Leb[d]{2}{2}$.
\end{proof}

\section{Convergence lemmas} \label{sec:appb}

A similar result to the following appeared in Kazhikhov \cite{kaz98} with
stronger assumptions. Here we give a proof of this ``compensated compactness'' lemma
by following the ideas in the proof of Droniou--Eymard \cite[Theorem 5.4]{DE15}.

\begin{theorem}\label{lem:cc}
For $\psi\in C^\infty_c(\O)$ and $\gamma\in L^1(\O\times (0,T))$, define
$F_\gamma^\psi\in L^1(0,T)$ by $F_\gamma^\psi(t)=\int_\O \gamma(x,t)\psi(x)\ud x$.
Let $a,p\in (1,\infty)$ and $(\alpha_\eps)_{\eps>0}$ and $(\beta_\eps)_{\eps>0}$ be sequences such that
\begin{align*}
&(\alpha_\eps)_{\eps>0} \mbox{ is bounded in }L^{p}(0, T; W^{1,a}(\O)),\\
&\alpha_\eps \rightharpoonup \alpha \mbox{ weakly in }\Leb{p}{a},\\
&\beta_\eps \rightharpoonup \beta \mbox{ weakly in }\Leb{p'}{a'} \mbox{ and}\\
&\forall \psi\in C^\infty_c(\O)\,,\;(F_{\beta_\eps}^\psi)_{\eps>0} \mbox{ is bounded in }
BV(0,T).
\end{align*}
Then up to a subsequence,
\begin{equation*}
\alpha_\eps\beta_\eps \rightharpoonup \alpha\beta \mbox{ in }\distt.
\end{equation*}
\end{theorem}

\begin{proof}
Let
\[
\mathcal A(W,Z)= \int_0^T \int_\O W(x,t)Z(x,t)\ud x\ud t.
\]
We prove that for every $\phi\in C^\infty_c(\O\times (0,T))$,
\begin{equation}\label{cc.toshow}
\mathcal A(\alpha_\eps\phi,\beta_\eps)\to \mathcal A(\alpha\phi,\beta).
\end{equation}
In this proof, $C$ denotes a generic constant that does not depend on $\eps$.

\emph{Step 1: reduction to tensorial functions.}

Consider a covering $(C_\ell^\delta)_{\ell=1,\ldots,M}$
of $\O$ in cubes of length $\delta$. For $g\in L^a(\RR^d)$, define
\[
R_\delta g= \sum_{\ell=1}^M \left(\frac{1}{\meas(C_\ell^\delta)}\int_{C_\ell^\delta} g(x)\ud x\right)
\mathbf{1}_{C_\ell^\delta},
\]
where $\mathbf{1}$ denotes the characteristic function. In what follows, take $g\in W^{1,a}_0(\O)$.
Using Jensen's inequality, a linear
change of variable, and a standard inequality for functions in $W^{1,a}_0(\O)$ (extended by $0$
outside $\O$),
\[
\norm{R_\delta g-g}_{L^a(\O)}\le C\sup_{z\in (-\delta,\delta)^d} \norm{g(\cdot+z)-g}_{L^a(\RR^d)}
\le C\delta \norm{g}_{W^{1,a}_0(\O)}.
\]
The sequence of functions $(\alpha_\eps\phi)_{\eps>0}$ is bounded in $L^p(0,T;W^{1,a}_0(\O))$,
the zero boundary condition coming from the support of $\phi$. Hence
\[
\norm{R_\delta (\alpha_\eps\phi)-(\alpha_\eps\phi)}_{\Leb{p}{a}}\le C\delta.
\]
By the weak convergence of $(\alpha_\eps)_{\eps>0}$, this estimate also holds with $\alpha_\eps$
replaced by $\alpha$.
Using the boundedness of $(\beta_\eps)_{\eps>0}$ in $\Leb{p'}{a'}$ and the H\"older inequality,
\begin{equation}\label{cv.A.1}
|\mathcal A(\alpha_\eps\phi,\beta_\eps)-\mathcal A(\alpha\phi,\beta)|
\le C\delta + |\mathcal A(R_\delta(\alpha_\eps\phi),\beta_\eps)-\mathcal A(R_\delta(\alpha\phi),\beta)|.
\end{equation}
For a fixed $\delta$, assume that we can prove
that
\begin{equation}\label{step1:toprove}
\mathcal A(R_\delta(\alpha_\eps\phi),\beta_\eps)\to \mathcal A(R_\delta(\alpha\phi),\beta)\quad
\mbox{ as $\eps\to0$}.
\end{equation}
Then taking the limit superior as $\eps\to0$ and then the limit as $\delta\to 0$
of \eqref{cv.A.1} would show that \eqref{cc.toshow} holds.

\emph{Step 2: reduction to smooth functions.}

By construction of $R_\delta$, we have
\[
\begin{aligned}
&R_\delta(\alpha_\eps\phi)(x,t)=\sum_{\ell=1}^M \xi_{\eps,\ell}(t)\mathbf{1}_{C_\ell^\delta}(x),\\
&\mbox{with}\quad
\xi_{\eps,l}(t)=\frac{1}{\meas(C_\ell^\delta)}\int_{C_\ell^\delta}\alpha_\eps(x,t)\phi(x,t)\ud x.
\end{aligned}
\]
Hence by the bilinearity of $\mathcal A$, \eqref{step1:toprove} follows if we can establish that
\begin{equation}\label{step2.toprove.1}
\mathcal A(\xi_{\eps,\ell}\otimes \mathbf{1}_{C_\ell^\delta},\beta_\eps)
\to \mathcal A(\xi_\ell\otimes \mathbf{1}_{C_\ell^\delta},\beta)
\quad\mbox{ as $\eps\to0$},
\end{equation}
where 
\[
\xi_\ell(t)=\frac{1}{\meas(C_\ell^\delta)}\int_{C_\ell^\delta}\alpha(x,t)\phi(x,t)\ud x.
\]
Let $\psi\in C^\infty_c(\O)$. Using the bounds on $(\alpha_\eps)_{\eps>0}$ and $(\beta_\eps)_{\eps>0}$ we have
\[
\left|\mathcal A(\xi_{\eps,\ell}\otimes \mathbf{1}_{C_\ell^\delta},\beta_\eps)-
\mathcal A(\xi_{\eps,\ell}\otimes \psi,\beta_\eps)\right|\leq C \norm{\mathbf{1}_{C_\ell^\delta}-\psi}_{L^a(\O)}
\]
where $C$ may depend on $\ell$ and $\delta$, but not on $\eps$. A similar estimate holds
with $\xi_\ell$ and $\beta$ instead of $\xi_{\eps,\ell}$ and $\beta_\eps$.
Since $\Vert \mathbf{1}_{C_\ell^\delta}-\psi\Vert_{L^a(\O)}$ can be made arbitrarily small by an appropriate
choice of $\psi$, we see that \eqref{step2.toprove.1} holds provided that, for any $\psi\in C^\infty_c(\O)$,
\[
\mathcal A(\xi_{\eps,\ell}\otimes \psi,\beta_\eps)\to \mathcal A(\xi_\ell\otimes \psi,\beta)
\quad\mbox{ as $\eps\to0$}.
\]

\emph{Step 3: conclusion.}

We have
\[
\mathcal A(\xi_{\eps,\ell}\otimes \psi,\beta_\eps)=\int_0^T \xi_{\eps,\ell}(t) F_{\beta_\eps}^{\psi}(t)\ud t.
\]
The weak convergence of $(\beta_\eps)_{\eps>0}$
ensures that $F_{\beta_\eps}^\psi\to F_\beta^\psi$ in $\mathcal D'(0,T)$. Since $(F_{\beta_\eps}^\psi)_{\eps>0}$
is bounded in $BV(0,T)$, this convergence also holds in $L^{p'}(0,T)$.
On the other side, the weak convergence of $(\alpha_\eps)_{\eps>0}$ shows that
$\xi_{\eps,\ell}\to \xi_\ell$ weakly in $L^p(0,T)$. Hence as $\eps\to0$,
\[
\mathcal A(\xi_{\eps,\ell}\otimes \psi,\beta_\eps)=\int_0^T \xi_{\eps,\ell}(t) F_{\beta_\eps}^{\psi}(t)\ud t\to 
\int_0^T \xi_{\ell}(t) F_{\beta}^{\psi}(t)\ud t
=\mathcal A(\xi_{\ell}\otimes \psi,\beta). \qedhere
\]
\end{proof}

\begin{lemma}\label{pt.bound}
Let $E$ be a Banach space, $T>0$, and $(f_m)_{m\in\NN}$ be a bounded sequence in $L^1(0,T;E)$. 
Then for almost every $t\in (0,T)$,
there exists a subsequence $(f_{m_k})_{k\in\NN}$ such that $(f_{m_k}(t))_{k\in\NN}$ is bounded in $E$.
\end{lemma}

\begin{proof}
Let $Z$ be the set of $t\in (0,T)$ such that no subsequence exists along which $\norm{f_m(t)}_E$ is bounded.
Then for every $t\in Z$, $\norm{f_m(t)}_E\to\infty$. Using Fatou's lemma and denoting by $M$ a bound
of $(f_m)_{m\in\NN}$ in $L^1(0,T;E)$ gives
\begin{align*}
\meas(Z) \times (\infty)={}&\int_A \liminf_{m\to\infty}\norm{f_m(t)}_E\ud t\\
\le{}&
\int_0^T\liminf_{m\to\infty}\norm{f_m(t)}_E\ud t\leq \liminf_{m\to\infty}\int_0^T \norm{f_m(t)}_E\ud t\le M.
\end{align*}
This shows that $\meas(Z)=0$. \end{proof}

The following two lemmas are proved in \cite{dt14}.

\begin{lemma} \label{lem:h}
Let $\O$ be a bounded subset of $\RR^{d}$, $d\in\NN$, and for $\eps>0$ 
let $H_\eps:\O\times\RR^d\to\RR$ be a Carath\'eodory function such that
\begin{itemize}
\item there exist positive constants $\ctel{hlem}$, $\gamma$ such that for a.e.\ $x\in\O$, 
\begin{equation*}
|H_\eps(x, \xi)|\leq\cter{hlem}(1 +|\xi|^{\gamma}) \mbox{\quad $\forall\xi\in\RR^{d}$, $\forall\eps>0$;}
\end{equation*}
\item there is a Carath\'eodory function $H:\O\times\RR^{d}\to\RR$ such that for a.e.\ $x\in\O$, 
\begin{equation*}
H_\eps(x,\cdot)\to H(x,\cdot)\mbox{\quad uniformly on compact sets as $\eps\to0$.}
\end{equation*} 
\end{itemize}
If $p$, $q\in [\max(1,\gamma),\infty)$ and $(u_\eps)_{\eps>0}\subset L^{p}(0,T; L^{q}(\O)^{d})$ is a sequence with $u_\eps \to u$ in $L^{p}(0,T; L^{q}(\O)^{d})$ as $\eps\to0$, then
$H_\eps(\cdot, u_\eps)\to H(\cdot,u)$ in $L^{p/\gamma}(0,T; L^{q/\gamma}(\O))$ as $\eps\to0$.
\end{lemma}

\begin{lemma} \label{lem:ws}
Let $\O$ be a bounded, open subset of $\RR^{d}$ and for $\eps>0$, 
let $w_\eps:\O\times(0,T)\to\RR$ and $v_\eps:\O\times(0,T)\to\RR$ be such that as $\eps\to0$,
\begin{align*}
w_\eps &\to w \mbox{\quad strongly in $L^{r_{1}}(0,T;L^{s_{1}}(\O))$, and}\\
v_\eps &\to v \mbox{\quad weakly in $L^{r_{2}}(0,T; L^{s_{2}}(\O))$,}
\end{align*}
where $r_{1}, r_{2}, s_{1}, s_{2} \geq 1$ are such that $1/r_{1} + 1/r_{2} \leq 1$ 
and $1/s_{1} + 1/s_{2} \leq 1$. Suppose also that the sequence $(w_\eps v_\eps)_{\eps>0}$ 
is bounded in $L^{a}(0,T; L^{b}(\Omega))$, where $a,b \in (1,\infty)$. Then
$w_\eps v_\eps \to wv$ weakly in $L^{a}(0,T; L^{b}(\Omega))$.
\end{lemma}


\bibliography{peaceman-nodiff.bib}{}

\begin{thebibliography}{10}

\bibitem{az04}
Y.~Amirat and A.~Ziani.
\newblock Asymptotic behavior of the solutions of an elliptic-parabolic system
  arising in flow in porous media.
\newblock {\em Z. Anal. Anwendungen}, 23(2):335--351, 2004.

\bibitem{anderson84}
M.P. Anderson.
\newblock Movement of contaminants in groundwater: groundwater
  transport--advection and dispersion.
\newblock In National~Research Council, editor, {\em Groundwater
  Contamination}, Studies in Geophysics, pages 37--45. National Academy Press,
  Washington DC, 1984.

\bibitem{bg89}
L.~Boccardo and T.~Gallou{\"e}t.
\newblock Nonlinear elliptic and parabolic equations involving measure data.
\newblock {\em J. Funct. Anal.}, 87(1):149--169, 1989.

\bibitem{bg92}
L.~Boccardo and T.~Gallou{\"e}t.
\newblock Nonlinear elliptic equations with right-hand side measures.
\newblock {\em Comm. Partial Differential Equations}, 17(3-4):641--655, 1992.

\bibitem{ce99}
Z.~Chen and R.~Ewing.
\newblock Mathematical analysis for reservoir models.
\newblock {\em SIAM J. Math. Anal.}, 30(2):431--453 (electronic), 1999.

\bibitem{dsws94}
B.~Das, S.~Steinberg, S.~Weber, and S.~Schaffer.
\newblock Finite difference methods for modeling porous media flows.
\newblock {\em Transport in Porous Media}, 17(2):171--200, 1994.

\bibitem{D03}
J.~Droniou.
\newblock Global and local estimates for nonlinear noncoercive elliptic
  equations with measure data.
\newblock {\em Comm. Partial Differential Equations}, 28(1-2):129--153, 2003.

\bibitem{DE15}
J.~Droniou and R.~Eymard.
\newblock Uniform-in-time convergence of numerical methods for non-linear
  degenerate parabolic equations.
\newblock {\em Numer. Math.}, 132(4):721--766, 2016.

\bibitem{dt14}
J.~Droniou and K.S. Talbot.
\newblock On a miscible displacement model in porous media flow with measure
  data.
\newblock {\em SIAM J. Math. Anal.}, 46(5):3158--3175, 2014.

\bibitem{ew83}
R.E. Ewing.
\newblock Problems arising in the modeling of processes for hydrocarbon
  recovery.
\newblock In R.E. Ewing, editor, {\em The Mathematics of Reservoir Simulation},
  Frontiers in Applied Mathematics, pages 3--34. SIAM, Philadelphia, 1983.

\bibitem{fg00}
P.~Fabrie and T.~Gallou\"et.
\newblock Modelling wells in porous media flow.
\newblock {\em Math. Models Methods Appl. Sci.}, 10(5):673--709, 2000.

\bibitem{fe95}
X.~Feng.
\newblock On existence and uniqueness results for a coupled system modeling
  miscible displacement in porous media.
\newblock {\em J. Math. Anal. Appl.}, 194(3):883--910, 1995.

\bibitem{mg99}
T.~Gallou\"et and A.~Monier.
\newblock On the regularity of solutions to elliptic equations.
\newblock {\em Rend. Mat. Appl. (7)}, 19(4):471--488 (2000), 1999.

\bibitem{gr85}
P.~Grisvard.
\newblock {\em Elliptic problems in nonsmooth domains}, volume~24 of {\em
  Monographs and Studies in Mathematics}.
\newblock Pitman (Advanced Publishing Program), Boston, MA, 1985.

\bibitem{kaz98}
A.V. Kazhikhov.
\newblock {\em Recent developments in the global theory of two-dimensional
  compressible {N}avier-{S}tokes equations}, volume~25 of {\em Seminar on
  Mathematical Sciences}.
\newblock Keio University, Department of Mathematics, Yokohama, 1998.

\bibitem{kov63}
E.J. Koval.
\newblock A method for predicting the performance of unstable miscible
  displacement in heterogeneous media.
\newblock {\em Soc. Pet. Eng. J.}, 3(2):145--154, 1963.

\bibitem{pe66}
D.W. Peaceman.
\newblock Improved treatment of dispersion in numerical calculation of
  multidimensional miscible displacement.
\newblock {\em Soc. Pet. Eng. J.}, 6(3):213--216, 1966.

\bibitem{pe77}
D.W. Peaceman.
\newblock {\em Fundamentals of Numerical Reservoir Simulation}.
\newblock Elsevier, New York, 1977.

\bibitem{pt02}
J.R.A. Pearson and P.M.J. Tardy.
\newblock Models for flow of non-newtonian and complex fluids through porous
  media.
\newblock {\em J. Non-Newtonian Fluid Mech.}, 102(2):447--473, 2002.

\bibitem{rw83}
T.F. Russell and M.F. Wheeler.
\newblock Finite element and finite difference methods for continuous flows in
  porous media.
\newblock In R.E. Ewing, editor, {\em The Mathematics of Reservoir Simulation},
  Frontiers in Applied Mathematics, pages 34--106. SIAM, Philadelphia, 1983.

\bibitem{bs84}
G.R. Shubin and J.B. Bell.
\newblock An analysis of the grid orientation effect in numerical simulation of
  miscible displacement.
\newblock {\em Comp. Meth. in Appl. Mech. \& Eng.}, 47(1):47--71, 1984.

\bibitem{sim-87-com}
J.~Simon.
\newblock Compact sets in the space {$L^p(0,T;B)$}.
\newblock {\em Ann. Mat. Pura Appl. (4)}, 146:65--96, 1987.

\bibitem{stam65}
G.~Stampacchia.
\newblock Le probl\`eme de {D}irichlet pour les \'equations elliptiques du
  second ordre \`a coefficients discontinus.
\newblock {\em Ann. Inst. Fourier (Grenoble)}, 15(fasc. 1):189--258, 1965.

\bibitem{young84}
L.C. Young.
\newblock A study of spatial approximations for simulating fluid displacements
  in petroleum reservoirs.
\newblock {\em Comp. Methods in Appl. Mech. \& Eng.}, 47(1-2):3--46, 1984.

\end{thebibliography}
\bibliographystyle{plain}

\end{document}